\documentclass[12pt]{amsart}
\usepackage{graphicx}
\usepackage{amsmath,graphics}
\usepackage{tikz}
\usepackage{amsfonts,amssymb}
\usepackage{xypic}
\usepackage{comment}
\usepackage{color}
\usepackage{hyperref}
\usepackage{cleveref}
\usepackage{thmtools}
\usepackage{thm-restate}
\usepackage{mathtools}
\usetikzlibrary{braids}

\specialcomment{proofc}{}{}
\includecomment{proofc}
\theoremstyle{plain}
\newtheorem*{theorem*}{Theorem}
\newtheorem*{lemma*} {Lemma}
\newtheorem*{corollary*} {Corollary}
\newtheorem*{proposition*}{Proposition}
\newtheorem*{conjecture*}{Conjecture}
\newtheorem{theorem}{Theorem}[section]
\newtheorem{lemma}[theorem]{Lemma}
\newtheorem*{theorem1*}{Theorem 1}
\newtheorem*{theorem2*}{Theorem 2}
\newtheorem*{theorem3*}{Theorem 3}
\newtheorem{corollary}[theorem]{Corollary}
\newtheorem{proposition}[theorem]{Proposition}

\newtheorem{question}[theorem]{Question}

\theoremstyle{remark}
\newtheorem*{remark}{Remark}
\newtheorem*{remarks}{Remarks}

\newtheorem*{example*}{Example}
\newtheorem*{claim}{Claim}

\theoremstyle{definition}

\def\op{\operatorname}

\def\G{\Gamma}

 \def\Q{\mathbb{Q}}  \def\Z{\mathbb{Z}} \def\R{\mathbb{R}}  
    
  \def\b{\beta}  \def\bp{\begin{pmatrix}}
 \def\ep{\end{pmatrix}} \def\bn{\begin{enumerate}} 
   \def\en{\end{enumerate}}
\def\ba{\begin{array}} \def\ea{\end{array}}  
  \def\b{\beta}  
    
  \def\Ker{\operatorname{Ker}}
\def\ker{\op{Ker}}\def\be{\begin{equation}} \def\ee{\end{equation}}
\def\An{\op{Aut}(F_n)} \def\A2{\op{Aut}(F_2)}  
\def\On{\op{Out}(F_n)} \def\O2{\op{Out}(F_2)}
 \def\SA2{\op{Aut^{+}}(F_2)} 
\def\IA2{\op{Inn}(F_2)} 
 \def\SO2{\op{Out^{+}}(F_2)}
\def\St{\op{Stab}(H)} \def\Sy4{\op{S}_4}

 \def\aut{\mbox{Aut}}

\def\op{\operatorname}
\def\what{\widehat}
\def\fin{\unlhd_f}
\def\s{\sigma_} \def\u{\sigma^{-1}_}

\begin{document}

\title[Higher incoherence of the automorphism groups of a free group]{Higher incoherence of the automorphism groups of a free group}

\author{Stefano Vidussi}
\address{Department of Mathematics, University of California, Riverside, CA 92521, USA} \email{svidussi@ucr.edu}
\date{\today}

\begin{abstract} Let $F_n$ be the free group on $n \geq 2$ generators. We show that for all $1 \leq m \leq 2n-3$ (respectively for all $1 \leq m \leq 2n-4$) there exists a subgroup of $\An$ (respectively $\On$) which has finiteness of type $F_{m}$ but not of type $FP_{m+1}(\Q)$, hence it is not $m$--coherent.
In both cases, the new result is the upper bound $m= 2n-3$ (respectively $m = 2n-4$), as it cannot be obtained by embedding direct products of free noncyclic groups, and certifies higher incoherence up to the virtual cohomological dimension and is therefore sharp. As a tool of the proof, we discuss the existence and nature of multiple inequivalent extensions of a suitable finite-index subgroup $K_4$ of $\A2$ (isomorphic to the quotient of the pure braid group on four strands by its center): the fiber of four of these extensions arise from the strand-forgetting maps on the braid groups, while a fifth is related with the Cardano--Ferrari epimorphism.

\end{abstract}

\maketitle

\section{Introduction}

A group $G$ is called {\em coherent} if all its finitely generated subgroups are finitely presented. Following \cite{KV23}, we say that a group is $m$--coherent if all its subgroups of type $F_m$ are of type $F_{m+1}$. If $H \leq G$ is a group of type $F_m$ that fails to be of type $F_{m+1}$, we say that $H$ is a {\em witness} to $m$-incoherence. Free groups, surface groups, and $3$--manifold groups give examples of groups that are $m$--coherent for all $m$, but it is a classical result that the product of two free non-cyclic groups, e.g. $F_2 \times F_2$, is not $1$-coherent (i.e. coherent). Iterating this fact, one can prove that the direct product $(F_2)^n$ of $n$ copies of $F_2$ fails to be $m$--coherent for all $1 \leq m \leq n-1$. The group $(F_2)^n$ is an example of {\em poly-free} group of length $n$, namely it admits a subnormal filtration of length $n$ whose successive quotients are finitely generated free groups, and the (in)coherence result mentioned above extends, under appropriate circumstances, to that class of groups, see \cite[Theorem 1.7 and Corollary 1.10]{KV23}. Remarkably, in all those examples, one can exhibit as witness to $m$--incoherence an {\em algebraic fiber}, i.e. a normal finitely generated subgroup which appears as kernel of a discrete character $\varphi \colon G \to \Z$, referred to as an {\em algebraic fibration}. More generally, given a group $G$, one could hope such a phenomenon to hold at least virtually, i.e. for a finite index subgroup of $G$. In fact, this is the startegy employed in \cite{KW22} to show that $\A2$ is incoherent. (A previous proof of this fact, based on a different approach, appears in \cite{G04}.) But such a strategy is doomed to fail for other groups, e.g. $\An$ or $\On$, respectively the automorphism and the outer automorphism groups of the free group $F_n$, for $n \geq 4$, as these groups do not admit, even virtually, algebraic fibrations as their virtual first Betti numbers vanish. A different approach to prove $m$-incoherence of a group $G$ is to try to embed in $G$ a group which is itself $m$-incoherent, and this is what we will endeavor to do here for $\An$ and $\On$. For the first, it is known that it admits as  subgroup the direct product $(F_2)^{2n-3} \leq \An$. This implies that for all $1 \leq m \leq 2n-4$, $\An$ is $m$-incoherent. Furthermore, by \cite[Theorem 6.1]{HW20}, $(F_2)^{2n-3}$ is the ``largest" (in terms of factors) direct product of free noncyclic groups contained in $\An$ for $n \geq 2$, so one cannot use other direct products to improve on this result. In a similar vein,  again by \cite[Theorem 6.1]{HW20}, the ``largest" direct products of free noncyclic groups contained in $\On$ for $n \geq 3$ are subgroups of the form $(F_2)^{2n-4} \leq  \On$, which entails that for all $1 \leq m \leq 2n-5$, $\On$ is $m$-incoherent. We will top the upper bounds by using a convenient embedding of non-product poly-free groups in $\An$ and $\On$ discussed in \cite{BKK02} and show the following:
\begin{restatable}{theorem}{mainthm}\label{thm:main} For all $n \geq 3$, there exists a (non-direct product) poly-free subgroup $F_n \rtimes (F_2^{2n-4} \rtimes \mathbb{F}_2) \leq \An$ of length $2n-2$ (respectively $F_2^{2n-4} \rtimes \mathbb{F}_2 \leq \On$ of length $2n-3$) which admit a virtual algebraic fiber of type $F_{2n-3}$ but not of type $FP_{2n-2}(\Q)$ (respectively of type $F_{2n-4}$ but not of type $FP_{2n-3}(\Q)$). In particular, $\An$ is not $(2n-3)$-coherent (respectively $\On$ is not $(2n-4)$-coherent). 
\end{restatable}

To the best of our knowledge, this result is new, and puts $\An$ and $\On$, as far as incoherence goes, at par with the direct product of free groups with the same virtual cohomological dimension. For instance, when $n=3$, this means that $\op{Out}(F_3)$ is not $2$--coherent. (The case  of coherence is already dispatched by an embedding of $F_2 \times F_2$, but, as observed before, $F_2 \times F_2 \times F_2 \nleq \op{Out}(F_3)$.)
Note that the result of Theorem \ref{thm:main} is sharp in the following sense: as the virtual cohomological dimension of $\An$ and $\On$ are respectively $2n-2$ and $2n- 3$ (see \cite{CV86,V02}), a subgroup with finiteness type $F_m$ up to the virtual cohomological dimension would have type $FP_{\infty}$, hence $F_{\infty}$ (see \cite[Proposition 6.1]{Br94}). Similar results of higher incoherence hold also for the Torelli subgroups of $\An$ and $\On$, building on \cite{GN21}, see Section \ref{sec:poly}.

Interestingly enough, Theorem \ref{thm:main} is a consequence of the fact that $\A2$ virtually algebraically fibers, a fact first established in \cite[Section 6]{KW22} and that essentially accounts to the case $n=2$ of Theorem \ref{thm:main}. ($\O2 \cong GL_2(\Z)$, hence it is virtually free.) We will show some alternative approaches to that fact, which will lead us to some related results on finite--index subgroups of $\A2$. 

Recall that by \cite{DFG82} the group $\SA2$ of special automorphisms of $F_2$, an index-$2$ subgroup of $\A2$, can described as the quotient $B_4/Z(B_4)$ of the braid group on $4$ strands by its center, hence it contains a normal finite index subgroup $K_4 \fin \A2$ that can be defined as the quotient $P_4/Z(P_4)$ of the braid group on $4$ strands by its center (see {\em Proof 2} of Theorem \ref{thm:kw} for details). As such, $K_4$ inherits $4$ inequivalent extensions as $F_3$-by-$F_2$ group from the strand-forgetting epimorphisms $\Theta_i \colon P_4 \to P_3$. At first sight, it may seems unlikely that $K_4$ admits other extensions. However we have the following. 

\begin{restatable}{theorem}{secthm}\label{thm:sec} Let $K_4  \leq_f \A2$ be the subgroup of index $48$ in $\A2$ defined as $K_4 = P_4/Z(P_4)$; then $K_4$ admits $5$ inequivalent extensions as an $F_3$-by-$F_2$ group, namely the fibers of these extensions are distinct normal subgroups of $K_4$; only $4$ of them arise from strand-forgetting maps on $P_4$. 
\end{restatable}

In Section \ref{sec:mul} we will argue that, in fact, the interesting part of Theorem \ref{thm:sec} is not the existence, but the origin of the fifth extension, as it relates to an ``unusual" (i.e. not a strand-forgetting) epimorphism $\Psi \colon P_4 \to P_3$ with kernel $F_5$, restriction of the so-called Cardano--Ferrari epimorphism. 

Furthermore, the fibers of these extensions can be described in terms of kernels of Birman's sequences associated with distinct identifications of $K_4$ with (pure) mapping class groups of punctured surfaces, and in particular they are related through the identification of $K_4$ with the pure mapping class group of a five-puntured sphere.

\subsection*{Acknowledgment.} The author would like to thank Macarena Arenas and Rob Kropholler for pointing out that the results of Theorem \ref{thm:main} are sharp, and Ric Wade for several mail exchanges that elucidated many properties of $\A2$ and significantly impacted the content of Section \ref{sec:mul}. Finally, we would like to thank the anonymous referee for their thorough review and suggestions, which greatly improved our presentation.

\section{Virtual fibrations of $\A2$} \label{sec:virt}

The strategy of the proof of Theorem \ref{thm:main} stems from the existence, for all $n \geq 3$, of poly-free groups of lengths $2n-2$ and $2n-3$ that embed as (infinite-index) subgroups of $\An$ and $\On$ respectively and for which we will prove the existence of virtually algebraic fibers with suitable finiteness property. We will be able to construct these groups by leveraging on the analog property for $\A2$.

The group $\A2$ can be written as an extension \begin{equation} \label{eq:autf2} 1 \longrightarrow \operatorname{Inn}(F_2) \longrightarrow \A2 \stackrel{\psi}{\longrightarrow} \O2 \longrightarrow 1 \end{equation} where we can identify the subgroup of inner automorphisms $\operatorname{Inn}(F_2)$ with $F_2$, as the latter has trivial center.
Note that $\O2 \cong GL_2(\Z)$. In particular, the natural epimorphism $\psi \colon \A2 \to \O2$ plays simultaneously the role of homological monodromy map, so that its kernel $F_2$ equals the Torelli subgroup $\operatorname{IA}(F_2)$. It is slightly more manageable to work with the  special automorphisms group $\SA2$, an index $2$ subgroup of $\A2$ containing $\operatorname{Inn}(F_2)$, and the corresponding special outer automorphism group; in particular $\SO2 \cong SL_2(\Z)$. (In fact, a simple way to define $\SA2$ is as $\psi^{-1}(SL_2(\Z))$.) It is well--known that $SL_2(\Z)$ is virtually free, hence in particular $vb_1(\SA2) = \infty$. What is not {\em a priori} obvious is the fact that the sequence in Equation (\ref{eq:autf2}) has virtually excessive homology, namely that there exists a finite index subgroup of $\A2$ with the property that the rank of the homology is strictly greater than that of the base of the resulting extension. This fact was first established in \cite[Section 6]{KW22} and we will outline their proof, adding some details that will be useful in what follows, and give another ({\em recte}, one and a half) somewhat different proof of this fact which has more geometric flavor. 
\begin{theorem} (Kropholler--Walsh) \label{thm:kw} The extension \[ 1 \longrightarrow \operatorname{Inn}(F_2) \longrightarrow \A2 \stackrel{\psi}{\longrightarrow} \O2 \longrightarrow 1\] has virtually ecessive homology.  \end{theorem}
\begin{proof}[Proof 1]
We will prove this result by showing that it holds for the induced extension $1 \to \operatorname{Inn}(F_2) \to \SA2 \to SL_2(\Z) \to 1$ of the special automorphism groups $\SA2$. 

Consider the (special) automorphisms of $F_2 = \langle a,b  \rangle $ defined as follows. 
\begin{align*} 
\lambda \colon  & a \mapsto ab & \rho \colon & a \mapsto a   \\ 
& b \mapsto b & & b \mapsto ba     
\end{align*}

Let $\pi \colon F_2 \to \Z_2$ be the epimorphism given by $\pi(a) = 0; \,\ \pi(b) = 1$ and denote $H \fin F_2$ be the kernel of $\pi$. It is straightforward to verify that $H$ is a free group on the three generators $x := a, \,\ y := b^2, \,\ z := b a b^{-1}$. In general, given a finite quotient $\pi \colon F_2 \to Q$ with kernel $R \fin F_2$, we can consider the stabilizer \[ \op{Stab}(R) := \{f \in \A2 \, | \,  f(R) = R \} \leq_f \A2 \] and the {\em standard congruence subgroup} \[ \G(Q,\pi) := \{f \in \A2 \, |\, \pi \circ f = \pi \} \fin \op{Stab}(R) \] which can be defined as kernel of the induced map $\op{Stab}(R) \to \op{Aut}(Q)$. When $\ker{\pi} = R$ is characteristic, $\op{Stab}(R) \cong \A2$ and we refer to $\G(Q,\pi)$ as a {\em principal congruence subgroup}. We will denote \[ \G^{+}(Q,\pi) := \G(Q,\pi) \cap \SA2\] and refer to this group as standard congruence subgroup as well. It is quite straightforward to see that $\G^{+}(Q,\pi)$ is finite index in $\SA2$.  

Determining explicitly $\op{Stab}(\ker{\pi})$ or 
$\G^{+}(Q,\pi)$ is however rarely easy. In \cite[Proposition 6.2]{GL09} the authors determine that, for any epimorphism $\pi \colon F_2 \to \Z_2$, the index of $\G^{+}(\Z_2,\pi) = \op{Stab}(R) \cap \SA2 \leq_f \SA2$ in the special automorphism group $\SA2$ is $3$. (Recall that $\op{Aut}(\Z_2)$ is trivial.) In order to determine $\G^{+}(\Z_2,\pi) \leq_f \SA2$ to the extent we need, we will proceed as follows. Start by considering the four elements of $\SA2$ defined as 
\begin{equation} \label{eq:fourgen}  \lambda^2, \rho, \lambda^{-1} \rho^2 \lambda, \lambda^{-1} \rho^{-1} \lambda \rho \lambda \in \SA2 \end{equation} (Here we follow the convention that composition of automorphism goes from the left to the right, e.g. $\lambda \rho (a) = \rho(ab) = aba$, so our notation is the transpose of that of \cite{KW22}.) By explicit calculation, one can verify that these elements are contained in $\St$. Next, consider their images under $\psi  \colon \SA2 \to \SO2 \cong SL_2(\Z)$. Using an explicit presentation of the latter group it is not too difficult to verify, using GAP, that their image generates a subgroup $S \leq_f SL_2(\Z)$ of index $3$. The theory of presentation of group extensions guarantees then that the collection of automorphisms of $F_2$ given by the four elements in Equation (\ref{eq:fourgen}) plus the generators given by the inner automorphisms $i_a,i_b \in \op{Inn}(F_2) \unlhd \SA2$ generate a subgroup of index $3$ in $\SA2$ which is an extension of $S$ by $F_2$;  as $H \fin F_2$, inner automorphisms of $F_2$ stabilize $H$, i.e. $\operatorname{Inn}(F_2) \leq \St$. It follows that this subgroup of index $3$ must actually coincide with $\G^{+}(\Z_2,\pi)$, so we have 
\[ 1 \longrightarrow F_2 \longrightarrow \G^{+}(\Z_2,\pi) \longrightarrow S \longrightarrow 1 \] 
Next, we define a subgroup $S(H)$ of finite index in $\G^{+}(\Z_2,\pi)$ which is an extension with fiber $H$ and base a subgroup of $S$: phrased otherwise, we are trying to find a subgroup $S(H) \leq_f \G^{+}(\Z_2,\pi)$ such that $S(H) \cap F_2 = H$. Using GAP, it turns out that the epimorphism $\pi \colon F_2 \to \Z_2$ does extend to an epimorphism (that we denote with the same symbol) $\pi \colon \G^{+}(\Z_2,\pi) \to \Z_2$ so we can define $S(H) := \ker{\pi} \fin \G^{+}(\Z_2,\pi)$, which is therefore an extension of $S$ by $H$; however, it is not clear to the author how to prove this directly. We can however obtain the result that we need using a more general approach, that provides us furthermore with one of the tools that we will need in the rest of this paper.
As $H \unlhd \G^{+}(\Z_2,\pi)$,  the quotient map $p \colon \G^{+}(\Z_2,\pi) \to \G^{+}(\Z_2,\pi)/H$ fits in the commmutative diagram of short exact sequences 
\begin{equation} \label{eq:diag} \xymatrix@R0.5cm{
& 1 \ar[d] &
1 \ar[d] & &\\
1 \ar[r] & H \ar[d] \ar[r] &
H \ar[d] \ar[r]& 1 \ar[d] & \\
1\ar[r]&
F_2 \ar[r]\ar[d]^{\pi}&
\G^{+}(\Z_2,\pi)\ar[d]^p \ar[r]& S \ar[d] \ar[r]&1\\
 1\ar[r]& \Z_2 \ar[d]
 \ar[r]& \G^{+}(\Z_2,\pi)/H \ar[r] \ar[d]  &
S \ar[r] \ar[d] &1 \\
& 1 & 1 & 1} 
\end{equation}
We focus on the bottom short exact sequence: we want to show that such extension is virtually a product. (The aforementioned GAP calculation amounts to say that it is actually a product.) We can observe that this extension is central (the group $\G^{+}(\Z_2,\pi)/H$ acts by conjugation on $\Z_2$, mapping onto $\operatorname{Aut}(\Z_2)$, but as the latter group is trivial, the conjugation action is trivial). Isomorphism classes of central extensions of $\Z_2$ by $S$ are classified by $H^2(S;\Z_2)$. Now, $S$ is virtually free, because so is $SL_2(\Z)$, hence there exists a free finite index subgroup $T \leq_f S$ for which $H^2(T;\Z_2) = 0$ and the pull-back of $\G^{+}(\Z_2,\pi)/H \to S$ to $T$ is a product, namely $\G^{+}(\Z_2,\pi)/H$ is virtually $\Z_2 \times T$. We have a commutative diagram
\[ \xymatrix@=10pt{ 
 1\ar[rr] & & F_2 \ar[rr]\ar'[d][dd] \ar[dr] & & p^{-1}(\Z_2 \times T) \ar'[d][dd]\ar[dr]  \ar[rr] & & T \ar[rr]\ar[dr]\ar'[d][dd] & & 1
\\ & 1\ar[rr]  &  & F_2 \ar[dd] \ar[rr] & & \G^{+}(\Z_2,\pi) \ar[rr]\ar[dd] &  & S \ar[rr]\ar[dd] & & 1
\\ 1 \ar[rr] &  & \Z_2 \ar'[r][rr]\ar[dr] & & \Z_2 \times T \ar'[r][rr]\ar[dr] & & T \ar'[r][rr] \ar[dr] & & 1
\\ & 1 \ar[rr] & & \Z_2 \ar[rr] & & \G^{+}(\Z_2,\pi)/H \ar[rr] & & S \ar[rr] & & 1 } \] where all diagonal maps are finite-index subgroup inclusions. Now the epimorphism $\pi \colon F_2 \to \Z_2$ extends to $\pi \colon p^{-1}(\Z_2 \times T) \to \Z_2$ via the projection map from $\Z_2 \times T$ to the first factor. Denote the kernel of this map by $T(H)$: we have $T(H) \leq_f p^{-1}(\Z_2 \times T) \leq_f \G^{+}(\Z_2,\pi)$, and as $F_2 \cap T(H) = H$, it is an extension of $T$ by $H$, namely
\begin{equation} \label{fig:ext} 1 \longrightarrow H \longrightarrow T(H) \longrightarrow T \longrightarrow 1 \end{equation} (This sequence splits, as $T$ is free.) We address the problem of showing that the sequence in Equation (\ref{fig:ext}) has excessive homology. A generating set of this group is given by the generators $i_x,i_y,i_z$ of $H \cong  \op{Inn}(H)$, which act trivially on $H_1(H;\Z)$ and a number of elements that lift a generating set of $T$. The action on $H_1(H;\Z)$ of these elements factors through the action of the subgroup of $\G^{+}(\Z_2,\pi)$ generated by the four generators listed in Equation (\ref{eq:fourgen}), now interpreted as elements of $\operatorname{Aut}(H)$. (The choice of the lift to $T(H)$ is immaterial, again because $H$ acts trivially on its homology.) The key point is that we can determine explicitly how these four generators act on $H$ and in particular how they act on $H_1(H;\Z)$. Explicit calculations (see \cite{KW22} for details) show that the coinvariant homology of $H$ under the action of the subgroup generated by the element in Equation (\ref{eq:fourgen}) has rank one, and this entail that $H_1(H;\Z)_{T(H)}$ has rank at least one. 
\end{proof}

\begin{proof}[Proof 2] The second proof is based on the existence, first pointed out in \cite{DFG82} and elucidated geometrically in \cite[Section 6]{BW24}, of an ``accidental" isomorphism $B_4/Z(B_4) \cong \SA2$, where $B_4$ is the braid group on $4$ strands. As well-known, the group $B_4$ admits the Artin presentation \begin{equation} 
\label{eq:presb4} B_4 \cong \langle \sigma_1,\sigma_2, \sigma_3|\sigma_1 \sigma_2 \sigma_1 = \sigma_2 \sigma_1 \sigma_2, \sigma_2 \sigma_3 \sigma_2 = \sigma_3 \sigma_2 \sigma_3, \sigma_1 \sigma_3 = \sigma_3 \sigma_1 \rangle \end{equation} The braid group contains a normal free group of rank two \[ F_2 \cong \langle a = \sigma_1 \sigma_3^{-1},b = \sigma_2 \sigma_1 \sigma_3^{-1}\sigma_2^{-1}\rangle \unlhd B_4\]  (see \cite{DFG82,GoLi69}), and $B_4$ acts by conjugation on $F_2$, inducing an homomorphism $B_4 \to \A2$. Obviously the action restricts trivially to the center $Z(B_4) = \langle (\sigma_1 \sigma_2 \sigma_3)^4 \rangle \cong \Z$, so the homomorphism $B_4 \to \A2$ descends to $B_4/Z(B_4)$. In \cite{DFG82} the authors show that the induced homomorphism is in fact injective and its image coincides with $\SA2$, so that  $B_4/Z(B_4) \cong \SA2$. Furthermore, the image of $F_2 \unlhd B_4$ (which has trivial intersection with $Z(B_4)$) in $\SA2$ is the subgroup $\operatorname{Inn}(F_2)$, as the conjugation action of $B_4$ on its $F_2$-subgroup restricts to the action of $F_2$ on itself by inner automorphisms. We will henceforth work with the presentation of $\SA2$ afforded by its isomorphism with $B_4/Z(B_4) $, which comes from the presentation of $B_4$ in Equation (\ref{eq:presb4}) by adding the center as relator:
\begin{equation} \label{eq:presaut} \SA2 \cong \langle \sigma_1,\sigma_2, \sigma_3|\sigma_1 \sigma_2 \sigma_1 = \sigma_2 \sigma_1 \sigma_2, \sigma_2 \sigma_3 \sigma_2 = \sigma_3 \sigma_2 \sigma_3, \sigma_1 \sigma_3 = \sigma_3 \sigma_1,  (\sigma_1 \sigma_2 \sigma_3)^4  \rangle \end{equation} Consider now the epimorphism $\Xi \colon B_4 \to \Sy4$ determined by the permutation of the strands (whose kernel is the pure braid group $P_4$); $\Xi$ descends as well to $B_4/Z(B_4)$, as $\Sy4$ has trivial center (or, if preferred, because one can observe that elements in $Z(B_4)$ are pure braids). The existence of the isomorphism $B_4/Z(B_4) \cong \SA2$ entails therefore that the latter has an (otherwise non-obvious) epimorphism $\xi \colon \SA2 \to \Sy4$. Considering that the center $Z(B_4)$ is contained in the pure braid subgroup, and it actually coincides with its center $Z(P_4)$, the kernel of this epimorphism is $K_4:= P_4/Z(P_4)$. (Note that it is well--known that $P_4 \cong K_4 \times Z(P_4)$, see \cite[Section 1.9.3]{FM12} or \cite[Corollary 3.6]{Lin04}.)  
Thanks to the identification $\operatorname{Inn}(F_2) \cong \langle \sigma_1 \sigma_3^{-1},\sigma_2 \sigma_1 \sigma_3^{-1}\sigma_2^{-1}\rangle$, we can explicitly identify $\xi(\IA2) = \Xi(F_2) \fin \Sy4$ just by determining how the generators permute strands: 
\[ \Xi(\sigma_1 \sigma_3^{-1}) = (1,2)(3,4) \in  \Sy4, \,\ \Xi(\sigma_2 \sigma_1 \sigma_3^{-1}\sigma_2^{-1}) = (1,3)(2,4) \in  \Sy4, \] and these are generators of the normal Klein subgroup $\xi(\IA2)  = \Xi(F_2) = \op{V}_4 \unlhd \Sy4$. We will denote $J := \ker\xi \cong F_5$. We thus have the commutative diagram
\begin{equation} \label{eq:pure} \xymatrix@R0.5cm{
& 1 \ar[d] &
1 \ar[d] & 1 \ar[d] &\\
1 \ar[r] & J \ar[d] \ar[r] &
K_4 \ar[d] \ar[r]^{\psi}& \G(2) \ar[d] \ar[r] & 1 \\
1\ar[r]&
\IA2 \ar[r]\ar[d]^{\xi}&
\SA2 \ar[d]^{\xi} \ar[r]^{\psi} & SL_2(\Z) \ar[d] \ar[r]&1\\
 1\ar[r]& \op{V}_4 \ar[d]
 \ar[r]& \Sy4 \ar[r] \ar[d]  &
\op{S}_3 \ar[r] \ar[d] &1 \\
& 1 & 1 & 1} 
\end{equation}
where we use the fact that the quotient $\Sy4/\op{V}_4 = \op{S}_3$.
The nature (and the notation) of the kernel $\G(2)$ of the epimorphism $SL_2(\Z)$ onto $\op{S}_3$ is explained in the following:
\begin{claim} Let $\chi \colon SL_2(\Z) \to S_3$ be an epimorphism: then the kernel $\ker \chi$ is the level-2 principal congruence subgroup \[ \G(2) := \left\{ \begin{pmatrix} a & b \\ c & c \end{pmatrix} \in SL_2(\Z), \begin{pmatrix} a & b \\ c & d \end{pmatrix} = \begin{pmatrix} 1 & 0 \\ 0 & 1 \end{pmatrix} \text{mod $2$} \right\}. \]
\end{claim}
\begin{proof}[Proof of Claim] 
It is known from \cite[Theorem 1]{A47} (see also \cite[Section 1.7]{Lin04}) that there exists a unique epimorphism from the braid group on $3$ strands $B_3$ to $\op{S}_3$, up to an automorphism of the latter. This implies that the same property holds for epimorphisms from any quotient of $B_3$, in particular $SL_2(\Z)$ (see \cite[Section 1.3.6.4]{FM12} for the construction of a quotient map). Because of this property, any epimorphism $SL_2(\Z) \to \op{S}_3$ has the same kernel. Now $\op{S}_3 \cong SL_2(\Z_2)$, and the \textit{mod $2$} reduction map from $SL_2(\Z)$ to $SL_2(\Z_2)$ is an epimorphism, whose kernel is the {\em principal congruence subgroup} $\G(2)$, which is well-known to be isomorphic to $F_2 \times \Z_2$.
\end{proof}  
(In fact we will show directly later on that $\G(2) \cong F_2 \times \Z_2$.) The Betti numbers of the pure braid groups are well--known, and out of that one can compute, in particular, that $b_1(K_4) = 5$. As $b_1(K_4) > b_1(\G(2)) = 2$ we conclude, perhaps unceremoniously, that the sequence in the first row of Equation (\ref{eq:pure}) has excessive homology. \end{proof}

\begin{remarks}
\begin{enumerate}
\item As observed in \cite[Section 1.9.2]{FM12}, $B_4/Z(B_4)$ can be identified with a subgroup of the mapping class group $\op{Mod}(S_{0,5})$ of the sphere with $5$ punctures; with this interpretation, $K_4$ is the pure mapping class group $\op{PMod}(S_{0,5})$. This will play a role in Section \ref{sec:mul}. In a similar vein, $B_4/Z(B_4)$ admits an identification with the pure mapping class group $\op{PMod}(S_{1,2})$ of the twice-puntured torus (see \cite[Proposition 6.1]{BW24}). It would probably be a nice exercise to recast another geometric proof of Theorem \ref{thm:kw} using this isomorphism. 
\item It is worth observing that $B_4$ is unique among braid groups insofar as it has, besides the ``natural" permutation epimorphism $\Xi \colon B_4 \to \Sy4$, two other epimorphisms onto $\Sy4$ which are not related by automorphisms of $\Sy4$ (see \cite[Theorem 1]{A47}, \cite[Sections 2 and 3]{MN22}). Again, these induce two epimorphisms $\SA2 \to \Sy4$, and one can verify, using GAP and techniques similar to those described above, that the sequences induced on the kernel of these epimorphisms as in Equation (\ref{eq:pure}) do not have excessive homology. (This observation has no bearing in what follow, so we omit the details.)
\end{enumerate}
\end{remarks}

We want to add some depth to the information about the the group $K_4 \fin \SA2$, and relate it with other groups of relevance. We start by relating $K_4$ with a principal congruence subgroup of $\SA2$.
\begin{proposition} \label{prop:stabaut} Let $\xi \colon F_2 \cong \op{Inn}(F_2) \to \op{V_4}$ be the epimorphism of Equation (\ref{eq:pure})); then the (principal) congruence subgroup $\G^{+}(\op{V_4},\xi) \fin \SA2$ is given by the subgroup $\langle K_4, a,b \rangle$, where $a,b \in F_2$ act by conjugation.
\end{proposition} 
\begin{proof} First, note that we can interpret $\xi$ as the {\em mod $2$} reduction of the maximal abelian quotient map \[ \xi \colon F_2 \cong \op{Inn}(F_2) \longrightarrow H_1(F_2;\Z_2) \cong \op{V_4} \] hence $J = \ker \xi \fin F_2$ is characteristic and $\op{Stab}(J) = \A2$. It follows from the definition of congruence subgroup (which we reviewed in the proof of Theorem \ref{thm:kw}) that $\G^{+}(\op{V_4},\xi) \fin \SA2$ is principal, and it equals the kernel of the map $\SA2 \to \op{Aut}(\op{V_4})$. We claim that this map is given by the composite map \begin{equation} \label{eq:stabaut} \SA2 \stackrel{\xi}{\longrightarrow} \op{S_4} \longrightarrow \op{S_3}  \cong \aut(\op{V_4})\end{equation} coming from the diagram of Equation (\ref{eq:pure}), where $\op{S_3}$ acts on $V_4$ via the conjugation action of $\op{S}_4$ on $V_4 \unlhd S_4$ which descends to $\op{S_3}$ as $\op{V_4}$ is abelian. (Note that $S_4$ can be seen as the holomorph $V_4 \rtimes \aut(V_4)$ of the Klein group.) In fact, $\xi \colon \SA2 \to \op{S_4}$ is induced by the permutation map $\Xi \colon B_4 \to \op{S}_4$, and the action of $\SA2$ on $F_2$ (given by conjugation by a braid in $B_4$ on $F_2 \unlhd B_4$) induces the action by conjugation by that braid's permutation in $\op{S}_4$ on $\op{V_4} \unlhd \op{S_4}$, hence the epimorphism in Equation (\ref{eq:stabaut}) induces the desired automorphism on the quotient $\op{V_4}$. Finally, it follows from the diagram in Equation (\ref{eq:pure}) that $\G^{+}(\op{V_4},\xi) = \xi^{-1}(\op{V}_4)$,  the kernel of the epimorphism in Equation (\ref{eq:stabaut}), is given by $\langle K_4, a,b \rangle$. 
 \end{proof}
\begin{remark} Note that we could have reached the same conclusion using \cite[Theorem 3.19]{Lin04}, which guarantees that there exists a unique epimorphism from $B_4$ to $\op{S_3}$ up to automorphism of the latter, hence the kernel of {\em any} epimorphism  $\SA2 \to \op{S_3}$ is the same. The proof above is more explicit, and tells us a bit more.
\end{remark}
The proof of Proposition \ref{prop:stabaut} entails that $\G^{+}(\op{V_4},\xi)$ is an extension of $F_2$ by $\G[2]$. Furthermore, as $\G^{+}(\op{V_4},\xi) = \xi^{-1}(\op{V}_4) \fin \SA2$, the epimorphism $\xi \colon F_2 \to \op{V}_4$ extends naturally to $\xi \colon \G^{+}(\op{V_4},\xi) \to \op{V_4}$ with kernel $K_4$. Consider the epimorphism \[ \G^{+}(\op{V_4},\xi) \to \G[2] \cong F_2 \times \Z_2 \to \Z_2,\] and denote its kernel $G_4$, a normal subgroup of index $12$ in $\SA2$. Observe that $\op{Inn}(F_2) \cong F_2 \unlhd G_4$ is the fiber of a $F_2$-by-$F_2$ extension whose base $F_2 := \ker(\G[2] \to \Z_2)$ is a normal subgroup of index $12$ in $SL_2(\Z)$, often referred to as the Sanov subgroup of $SL_2(\Z))$. 
We combine the observations above in the following:
\begin{proposition} There exists a commutative diagram  
\[ \xymatrix@=12pt{ 
J \ar[rrr] \ar[dr] \ar[ddd] & & & J \ar[dr] \ar@{-}'[d][dd]  & & & & &
\\ & \what{K}_4  \ar[rrr]\ar[ddd] \ar[dr]^{\psi}  &   & & K_4 \ar'[d][ddd] \ar[dr]^{\psi}  \ar[rrr] & &  & \Z_2 \ar[dr] \ar'[d][ddd] & 
\\&  & F_2 \ar[ddd] \ar[rrr] & \ar[d] & &  \G(2) \ar[rrr]\ar[ddd] & & & \Z_2 \ar[ddd]
\\ \IA2 \ar[ddd] \ar[dr] \ar@{-}'[r][rr] &  & \ar[r] & \IA2 \ar@{-}'[d][dd] \ar[dr] \ar[dr]  & & &  & &  &
\\ & G_4  \ar[ddd] \ar'[r][rrr] \ar[dr]^{\psi} &  &  & \G^{+}(\op{V_4},\xi) \ar'[d][ddd]  \ar'[r][rrr]\ar[dr]^{\psi} & &  & \Z_2 \ar[dr] &
\\&  & F_2 \ar[rrr]   & \ar[d] & & \G[2]  \ar[rrr] &  &  & \Z_2 & 
\\ \op{V}_4 \ar'[r][rrr] \ar[dr] & & &\op{V}_4 \ar[dr] &  &  & & & &
\\  & \op{V}_4 \ar[rrr] & &  & \op{V}_4  & & & & &
\\  & & &  & & & &  &  & } \] where all $2$- and $3$-term sequences are exact.
\end{proposition}

We emphasize that $G_4$ is an $F_2$-by-$F_2$ group, whose (fiberwise) index-$4$ subgroup $\what{K}_4$ is a free-by-free extension 
\[ 1 \longrightarrow J \longrightarrow \what{K}_4 \longrightarrow F_2 \longrightarrow 1. \] (According to \cite[Section 9]{GL09}, $H_1(\G^{+}(\op{V_4},\xi);\Z) = \Z^2 \times \Z_2^{3}$: the three $\Z_2$ factors are accounted in the diagram as $V_4$ and $\Z_2$.)

\begin{remark} The Sanov subgroup of $SL_2(\Z)$ is the subgroup generated by the matrices
\begin{align*}  (TST)^2 = & \begin{pmatrix} 1 & 0 \\ 2 & 1 \end{pmatrix} & T^2 = & \begin{pmatrix} 1 & 2 \\ 0 & 1 \end{pmatrix},
\end{align*}
where $S,T$ are standard generators of $SL_2(\Z)$, see Appendix.
\end{remark}

A consequence of Theorem \ref{thm:kw} is that $\A2$, as observed in \cite{KW22},  is not coherent. We present here a proof of this fact, which is essentially equivalent to the original one, and that will give us a slightly stronger result, as warm-up for the general case of Theorem \ref{thm:main}.  
\begin{corollary} (of Theorem \ref{thm:kw}) $\A2$ admits subgroups of type $F_1$ that are not of type $FP_2(\Q)$, in particular it is not coherent. 
\end{corollary} 
\begin{proof} Theorem \ref{thm:kw} entails  that $\A2$ is virtually algebraically fibered (see \cite[Theorem 5.1]{KW22}, \cite[Theorem 1]{FV23}), in particular $\what{K}_4$ (or even $K_4$), admits a fibration $\varphi \colon \what{K}_4 \to \Z$. Specifically, there exists a discrete character $\phi \colon J \to \R$ which is invariant under the action of $F_2$ (or even $\G[2]$) on $H^{1}(J;\Z)$. From the Lyndon--Hochschild--Serre sequence in cohomology, which takes the form \[ 0 \longrightarrow H^{1}(F_2;\Z) \longrightarrow H^{1}(\what{K}_4;\Z) \longrightarrow H^{1}(J;\Z)^{F_2} \longrightarrow 0 \] we deduce that $\phi \in H^{1}(J;\Z)^{F_2} \subset H^{1}(J;\Z)$ extends to a discrete character $\varphi \colon \what{K}_4 \to \R$ with kernel of type $F_1$. However, the fiber $\ker{\varphi}$ of this fibration cannot be of type $FP_2(\Q)$, hence {\em a fortiori} it is not finitely presented: in fact, $\what{K}_4$ is free-by-free, and if such fiber where $FP_2(\Q)$, we would have that the $L^{2}$-Betti numbers satisfy $b_{i}^{2}(\what{K}_4) = 0$ for $0 \leq i \leq 2$ by \cite[Theorem 7.2(5)]{Lu02}, see also \cite[Proposition 14]{IMP21}, all while $b_{2}^{(2)}(\what{K}_4) = \chi(\what{K}_4) = 4 \neq 0$ by general properties of poly-free groups (see \cite[Proposition 3.1]{GN21}). It follows that $\A2$ is not coherent in a strong form. 
\end{proof}

\section{Fibered poly-free subgroups of $\An$ and $\On$} \label{sec:poly}
We cannot hope to use (virtual) algebraic fibrations in the study of higher coherence of the groups $\An$ and $\On$, as these groups, at least for $n \geq 4$, have vanishing virtual first Betti numbers. However, we will up the results above to decide higher coherence properties for $\An$ and $\On$, and we will do so piggybacking on the result in Theorem \ref{thm:kw}. Consider again the Sanov subgroup $F_2 \fin \G[2] \fin SL _2(\Z)$. Without loss of generality, we can define a partial section of the map $\psi \colon \SA2 \to \SO2$ restricted to $F_2 \leq SL _2(\Z) \cong \SO2$, which amounts to embedding the base $F_2$ of $G_4$ as a subgroup of $\SA2$ that will be denoted $\mathbb{F}_2$, and to interpreting $G_4$ as a (specific) semi-direct product $F_2 \rtimes \mathbb{F}_2$.  (The fiber of $G_4$, properly, should be interpreted as $\IA2$, but we will omit this for sake of notation.) Furthermore, changing the partial section of $\psi$ if needed, we can also assume without loss of generality that $\mathbb{F}_2 \leq G_4$ is contained in the kernel $\what{K}_4$ of the map $\xi \colon G_4 = F_2 \rtimes \mathbb{F}_2 \to V_4$. This entails that $\what{K}_4 = \ker{\xi}$ admits as well the structure of semi-direct product $\what{K}_4 = J \rtimes \mathbb{F}_2$ with the action of $\mathbb{F}_2$ on $J$ given by the restriction of the action of $\mathbb{F}_2$ on the fiber of $G_4 = F_2 \rtimes \mathbb{F}_2$. 

Next, we consider the diagonal action of $\mathbb{F}_2 \leq \SA2$ on $F_2^{2n-4}$ to define the semi-direct product $F_2^{2n-4} \rtimes \mathbb{F}_2$. Our main interests in this construction is the fact that we can embed the group $F_2^{2n-4} \rtimes \mathbb{F}_2$ in both $\An$ and $\On$: this was proven in
\cite[Section 4]{BKK02}, see also \cite[Section 3]{GN21}:
\begin{proposition} \label{prop:bkk} (Bestvina--Kapovich--Kleiner) For all $n \geq 3$ there exists a monomorphism 
\[ F_2^{2n-4} \rtimes \mathbb{F}_2 \stackrel{\Phi}{\longrightarrow} \An \stackrel{\upsilon}{\longrightarrow} \On, \]
 hence the group $F_2^{2n-4} \rtimes \mathbb{F}_2$ is a subgroup of $\An$ and $\On$. \end{proposition}
\begin{proof} Recall that $\mathbb{F}_2$ is interpreted as a subgroup of $\SA2$, acting on $F_2 = \langle a,b \rangle$ hence diagonally on $F_2^{2n-4}$. We define a monomorphism $\Phi \colon F_2^{2n-4} \rtimes \mathbb{F}_2 \longrightarrow \An$ as follow. Identify $F_n := \langle a,b,x_3,\ldots,x_n \rangle$. For any $\kappa := (l_3,r_3,\ldots,l_n,r_n,\mu) \in F_2^{2n-4} \rtimes \mathbb{F}_2$ where the $l_i,r_i$ are words in $a$ and $b$ and $\mu \in \mathbb{F}_2 \leq \SA2$, we define an automorphism $\Phi_{\kappa} \colon F_n  \to F_n$ as 
\begin{align} 
\nonumber \Phi_{\kappa} \colon  & a \mapsto \mu(a)  \\ 
\label{eq:triple} & b \mapsto \mu(b)  \\
\nonumber & x_i \mapsto l_i^{-1} x_i r_i,  \,\ i=3,\ldots n.
\end{align} (Note that our definition differs, albeit in an immaterial way, from that of \cite{BKK02}.) As $\mathbb{F}_2$ acts diagonally on the $F_2$ factors, the product on $F_2^{2n-4} \rtimes \mathbb{F}_2$ is defined as
\[ \kappa \cdot \kappa' = (l_3,r_3,\ldots,l_n,r_n,\mu) (l'_3,r'_3,\ldots,l'_n,r'_n,\mu')  = (l_3 \mu (l'_3),r_3 \mu(r'_3),\ldots,l_n \mu (l'_n),r_n \mu(r'_n), \mu \circ \mu');\] this entails that $\Phi$ is a group homomorphism, as  
\begin{align*} 
 \Phi_{\kappa} \circ \Phi_{\kappa'} (a) & = \mu \circ \mu'(a) = \Phi_{\kappa \cdot \kappa'}(a) \\
 \Phi_{\kappa} \circ \Phi_{\kappa'} (b) & = \mu \circ \mu'(b) = \Phi_{\kappa \cdot \kappa'}(b) \\
 \Phi_{\kappa} \circ \Phi_{\kappa'} (x_i) & = \Phi_{\kappa}( (l'_i)^{-1} x_i r'_i) = (l_i \mu(l'_i)) ^{-1}  x_i r_i \mu(r'_i)  = \Phi_{\kappa \cdot \kappa'}(x_i)
\end{align*}
It is not difficult to see, from the definition, that $\Phi$ is a monomorphism. Furthermore, denoting $\upsilon \colon \An \to \On$ the natural epimorphism, if $\upsilon \circ \Phi_{\kappa} = 1_{\On}$, then $\Phi_{\kappa}$ must be an inner automorphism, which requires $l_i = r_i = c \in F_2 = \langle a,b\rangle, \,\ i = 3,\ldots,n$ and $\mu = i^{-1}_{c} \in \IA2$. But as by assumption the intersection $\mathbb{F}_2 \cap \IA2$ in the semi-direct product $G_4$ is the identity element, then $\kappa$ is the identity element in $F_2^{2n-4} \rtimes \mathbb{F}_2$.
\end{proof} 

\begin{corollary} \label{cor:inan} The action of $F_2^{2n-4} \rtimes \mathbb{F}_2$ on $F_n$ of Equation (\ref{eq:triple}) induces the commutative diagram
\begin{equation} \label{eq:inan} \xymatrix@R0.5cm{
& 1 \ar[d] &
1 \ar[d] & 1 \ar[d] &\\
1 \ar[r] & F_n \ar[d]^{\cong} \ar[r] &
F_n \rtimes (F_2^{2n-4} \rtimes \mathbb{F}_2) \ar[d] \ar[r] & F_2^{2n-4} \rtimes \mathbb{F}_2 \ar[d]^{\upsilon \circ \Phi} \ar[r] & 1 \\
1\ar[r]&
F_n \ar[r] &
\An \ar[r]^{\upsilon} & \On  \ar[r]&1\\} 
\end{equation}
\end{corollary}

\begin{proof} The epimorphism $\upsilon \colon \An \to \On$, restricting the codomain to $\upsilon \circ \Phi(F_{2}^{2n-4} \rtimes \mathbb{F}_2) \leq \On$, admits an obvious section, mapping $\upsilon \circ \Phi (\kappa)$ to $\Phi(\kappa)$ for all $\kappa \in F_{2}^{2n-4} \rtimes \mathbb{F}_2$. It follows that $F_n \rtimes (F_{2}^{2n-4} \rtimes \mathbb{F}_2) \leq \An$, where the semi-direct product structure is determined by the action of $F_{2}^{2n-4} \rtimes \mathbb{F}_2$ on the fiber $F_n$ described in Equation (\ref{eq:triple}). Note that, equivalently, we can define the extension in the first line of Equation (\ref{eq:inan}) as result of the pull-back of $\upsilon \colon \An \to \On$ under the inclusion $\upsilon \circ \Phi(F_{2}^{2n-4} \rtimes \mathbb{F}_2) \leq \On$, which is unique up to equivalence of extensions (\cite[Corollary IV.6.8]{Br94}). 
\end{proof}

\begin{remark} Note that the embeddings described in Corollary \ref{cor:inan} give poly-free subgroups of $\An$ and $\On$ of maximal length: in fact, the length of a poly-free groups equals their cohomological dimension (here $2n-2$ and $2n-3$ respectively), hence here it equals the virtual cohomological dimension of $\An$ and $\On$ respectively.
\end{remark}

Proposition \ref{prop:bkk} and Corollary \ref{cor:inan} are important for us in light of the following.
\begin{lemma} \label{lem:vaf} For $n \geq 3$,  consider the poly-free group  $F_2^{2n-4} \rtimes \mathbb{F}_2$ of length $2n-3$ and the poly-free group $F_n \rtimes (F_{2}^{2n-4} \rtimes \mathbb{F}_2)$ of length $2n-2$. These groups are virtually algebraically fibered and   \begin{enumerate} \item the group $F_2^{2n-4} \rtimes \mathbb{F}_2$ admits a virtual algebraic fibration whose fiber has type $F_{2n-4}$ but not type $FP_{2n-3}(\Q)$; \item the group $F_n \rtimes (F_{2}^{2n-4} \rtimes \mathbb{F}_2)$ admits a virtual algebraic fibration whose fiber has type $F_{2n-3}$ but not type $FP_{2n-2}(\Q)$. \end{enumerate} \end{lemma} 
\begin{proof} Recall that $G_4 \fin \G^{+}(\op{V}_4,\xi)$, hence the action of $\mathbb{F}_2 \leq \SA2$ on the fiber $F_2$ preserves the epimorphism $\xi \colon F_2 \to \op{V}_4$, namely $\xi(c) = \xi(\mu(c))$ for all $c \in F_2, \,\ \mu \in \mathbb{F}_2 \leq \SA2$. 
Therefore the diagonal action of $\mathbb{F}_2$ on $F_2^{2n-4}$ preserves the obvious induced epimorphism $\xi \colon F_2^{2n-4} \to \op{V}_4^{2n-4}$ (that we denote we the same symbol). We can extend this epimorphism to $F_2^{2n-4} \rtimes F_2$ without further ado: as the group in question is a semi-direct product, general theory of presentation of semi-direct products (see e.g. \cite[Section 10.3(S)]{Jo97} shows that we just need to map $\mathbb{F}_2$ to the trivial element of $\op{V}_4^{2n-4}$ to get a well-defined epimorphism 
\[ \xi \colon F_2^{2n-4} \rtimes \mathbb{F}_2 \to \op{V}_4^{2n-4},\] whose kernel is given by the finite index subgroup $J^{2n-4} \rtimes \mathbb{F}_2 \fin F_2^{2n-4} \rtimes \mathbb{F}_2$. Notice that the action of $\mathbb{F}_2$ by conjugation on each $J$ factor coincides with the action of $\mathbb{F}_2$ on the fiber $J$ of $\what{K}_4 = J \rtimes \mathbb{F}_2$. Let now $\phi \in H^1(J,\Z)^{\mathbb{F}_2}$ be the $\mathbb{F}_2$--invariant cohomology class on the fiber of $\what{K}_4$ afforded by the excessive homology of the latter. This class defines a (actually, a family of if one wishes so) cohomology class ${\what \phi} \in H^{1}(J^{2n-4};\Z)$, defined as sum of a copy of $\phi$ on each $J$ factor using the K\"unneth formula to relate $H^{1}(J^{2n-4};\Z)$ with $H^{1}(J;\Z)$, which is invariant under the diagonal $\mathbb{F}_2$ action. This class, thought of as a discrete character ${\what \phi} \colon J^{2n-4} \to \R$ has kernel of type $F_{2n-5}$: this follows e.g. from H. Meinart's inequality for BNSR invariants of direct products (see e.g. \cite[Theorem 1.2]{BG10}). By \cite[Theorem 1.1]{KV23} it extends to a discrete character ${\what \varphi} \colon J^{2n-4} \rtimes \mathbb{F}_2 \to \R$ with kernel of type $F_{2n-4}$. As $J^{2n-4} \rtimes \mathbb{F}_2$ is poly-free of length $2n-3$ and nonzero Euler characteristic, its $L^2$-Betti numbers satisfy $b_{i}^{(2)} = 0$ for $0 \leq i \leq 2n-4$ and  $b_{2n-3}^{(2)} \neq 0$ by \cite[Proposition 3.1]{GN21}, hence the kernel cannot have type $FP_{2n-3}(\Q)$ again by \cite[Proposition 14]{IMP21}. This proves the first part of the statement.

For the second part notice that, using $\xi \colon F_2^{2n-4} \rtimes \mathbb{F}_2 \to \op{V}_4^{2n-4}$, we can define an epimorphism $F_n \rtimes (F_{2}^{2n-4} \rtimes \mathbb{F}_2) \to F_{2}^{2n-4} \rtimes \mathbb{F}_2 \to \op{V}_4^{2n-4}$ which induces the finite index subgroup \[ F_n \rtimes (J^{2n-4} \rtimes \mathbb{F}_2) \fin F_n \rtimes (F_{2}^{2n-4} \rtimes \mathbb{F}_2). \] We want to apply to this group \cite[Theorem 1.7 and Remark 5.2]{KV23}, which assert in particular that given a group $G$ with normal filtration $1 = G_0 \unlhd G_1 \unlhd \ldots \unlhd G_m = G$ where all $G_i's$ are of type $F$, such that for all $0 \leq j \leq m-1$ the sequence
\begin{equation} \label{eq:exc} 1 \longrightarrow G_{j+1}/G_{j} \longrightarrow G/G_j \longrightarrow G/G_{j+1} \longrightarrow 1 \end{equation} has excessive homology, then there exists a discrete character ${\what \varphi} \colon G \to \R$, extending a discrete character ${\what \phi} \colon G_1 \to \R$, such that the kernel ${\what \varphi}$ has type $F_{m-1}$. We apply this result for $m = 2n-2$ to $G_m = G = F_n \rtimes (J^{2n-4} \rtimes \mathbb{F}_2)$, which admits a normal filtration
\[ 1 \unlhd F_n \unlhd F_n \rtimes J \unlhd F_n \rtimes J^2 \unlhd \ldots \unlhd F_n \rtimes J^{2n-4} \unlhd F_n \rtimes (J^{2n-4} \rtimes \mathbb{F}_2). \] The $2n-2$ sequences listed in Equation (\ref{eq:exc}) take the form
\begin{gather}  \nonumber 1 \longrightarrow F_n \longrightarrow F_n \rtimes (J^{2n-4} \rtimes \mathbb{F}_2) \longrightarrow J^{2n-4} \rtimes \mathbb{F}_2 \longrightarrow 1  \\
\nonumber  1 \longrightarrow J \longrightarrow J^{2n-4} \rtimes \mathbb{F}_2 \longrightarrow J^{2n-5} \rtimes \mathbb{F}_2 \longrightarrow 1 \\
\label{eq:gat} 1 \longrightarrow J \longrightarrow J^{2n-5} \rtimes \mathbb{F}_2 \longrightarrow J^{2n-6} \rtimes \mathbb{F}_2 \longrightarrow 1 \\
\nonumber \vdots \\
\nonumber  1 \longrightarrow J \longrightarrow J \rtimes \mathbb{F}_2 \longrightarrow F_2 \longrightarrow 1 \\
\nonumber 1 \longrightarrow F_2 \longrightarrow F_2 \longrightarrow 1 
\end{gather}
(Note that in the last two sequences, $F_2$ does not have to be thought of as subgroup of $F_n \rtimes (F_{2}^{2n-4} \rtimes \mathbb{F}_2)$ or its first $2n-4$ quotients, hence the different notation.) These sequences are all semi-direct products, hence their excessive homology is measured by the coinvariant homology of the fiber (see e.g. \cite[Lemma 2.1]{KVW23}), and we claim they all have excessive homology. For all but the first, this is a consequence of the fact that the sequence $1 \to J \to J \rtimes \mathbb{F}_2 \to F_2 \to 1$ itself has excessive homology because so does its supergroup $K_4$. For the first sequence, we can proceed as follows. We claim that the sequence \[ 1 \to F_n \to F_n \rtimes (F_2^{2n-4} \rtimes\mathbb{F}_2) \to F_2^{2n-4} \rtimes \mathbb{F}_2 \to 1 \] already satisfies 
$\operatorname{rk}(H_1(F_n;\Z)_{F_2^{2n-4} \rtimes \mathbb{F}_2}) = n-2$.
In fact, the homological monodromy $(\Phi_{\kappa})_{*} \colon H_1(F_n;\Z) \to H_1(F_n;\Z)$ for $\kappa := (l_3,r_3,\ldots,l_n,r_n,\mu) \in F_2^{2n-4} \rtimes \mathbb{F}_2$ has the matrix form
\[ \begin{pmatrix} \mu_{11} & \mu_{12} & 0 & \ldots & 0 \\ 
\mu_{21} & \mu_{22} & 0 & \ldots & 0 \\
\bar{r}_{3,a} - \bar{l}_{3,a} & \bar{r}_{3,b} - \bar{l}_{3,b} & 1 & \ldots & 0 \\
\vdots & \vdots & \vdots & \vdots  & 0 \\
\bar{r}_{n,a} - \bar{l}_{n,a} & \bar{r}_{n,b} - \bar{l}_{3,b} & 0 & \ldots & 1
\end{pmatrix} \] (acting on the left on row vectors) where $(\mu_{ij})$ are the entries of the matrix of $SL_{2}(\Z)$ corresponding to $\mu \in \mathbb{F}_2 \leq \SA2$ (by construction an element of the Sanov subgroup) and $\bar{l}_{i,a},\bar{r}_{i,a}$ and $\bar{l}_{i,b},\bar{r}_{i,b}$ are the images of $l_i,r_i \in F_2 =\langle a,b \rangle$ in \[ \Z_{a} \oplus \Z_{b} \subset \Z_{a} \oplus \Z_{b} \oplus \Z_{x_3} \oplus \ldots \oplus \Z_{x_n} = H_1(F_n;\Z). \]
From this form it should be clear that the span of $(\Phi_{\kappa})_{*} - I$ is a sublattice of $H_1(F_n;\Z)$ of rank at most $2$, hence the coinvariant homology \[ H_1(F_n;\Z)_{F_2^{2n-4} \rtimes \mathbb{F}_2} = H_1(F_n;\Z)/\langle \Phi_{\kappa}v -v, \forall \kappa \in F_2^{2n-4} \rtimes \mathbb{F}_2, \forall v \in H_1(F_n;\Z) \rangle \] has rank at least $n-2$. (A moment's thought shows it actually equals $n-2$.)

Next, observe that  the excessive homology of the first sequence in Equation (\ref{eq:gat}) satisfies $\operatorname{rk}(H_1(F_n;\Z)_{J^{2n-4} \rtimes \mathbb{F}_2}) \geq \operatorname{rk}(H_1(F_n;\Z)_{F_2^{2n-4} \rtimes \mathbb{F}_2}) > 0$, as coinvariant quotients are surjected upon by passing to the action of subgroups (we are ``quotienting less").

Summing up, \cite[Theorem 1.7]{KV23} applies and $F_n \rtimes (J^{2n-4} \rtimes \mathbb{F}_2)$ admits an algebraic fibration of type $F_{2n-3}$. Much as before, that fiber cannot have type $FP_{2n-2}(\Q)$ for Euler characteristic reasons combining again \cite[Proposition 3.1]{GN21} and \cite[Proposition 14]{IMP21}.  

 (Note that \cite[Theorem 1.7]{KV23} can be used also to prove the first part of the statement, which however affords the more direct proof given above.)
\end{proof}

This completes the proof of our main result:

\mainthm*

\begin{remarks} 
\begin{enumerate} 
\item In \cite[Section 4]{GN21} the authors modify the construction of Bestvina--Kapovich--Kleiner (Proposition \ref{prop:bkk} above) to identify subgroups $F_n \rtimes F_{2}^{2n-4} \leq \operatorname{IA}(F_n) \unlhd \An$ where $\operatorname{IA}(F_n)$ is the Torelli subgroup. The proof of Theorem \ref{thm:main} can then be applied, {\em mutatis mutandis}, to show that for all $1 \leq m \leq 2n-4$, $\operatorname{IA}(F_n)$ is not $m$-coherent. As the cohomological dimension of $\operatorname{IA}(F_n)$ equals $2n-3$ (see  \cite{BBM07}), this result is sharp as well. (The same happens for the kernel of $\On \to GL_n(\Z)$, but in this case the witnesses to incoherence are simply algebraic fibers of the direct product $F_{2}^{2n-4}$, so there is no need to appeal to \cite{KV23}.)
\item  In the proof of Theorem \ref{thm:main} (much as in \cite[Section 6]{KW22}) the role of excessive homology is to guarantee the existence of an algebraic fiber of a suitable poly-free group as witness to incoherence. Not all poly-free groups admit, even virtually, algebraic fibrations (see \cite{KVW23}) but it is conjectured that all poly-free groups exhibit the same (higher) incoherence properties of direct products (see \cite[Section 21, Problem 24]{W20} for the free-by-free case). If this were the case, the conclusions of Theorem \ref{thm:main} would follow directly from Proposition \ref{prop:bkk} without further ado.
\item The referee kindly pointed out that the group $F_n \rtimes (F_2^{2n-4} \rtimes \mathbb{F}_2) \leq \An$ contains an infinite index subgroup $F_2^{2n-3} \rtimes \mathbb{F}_2$ (where $\mathbb{F}_2$ acts diagonally) restricting the action of $(F_2^{2n-4} \rtimes \mathbb{F}_2) \An$ to $F_2 = \langle a,b \rangle \leq F_n$, as the action of $F_2^{2n-4}$ on that subgroup is trivial. This allows one to proof the second statement in Lemma \ref{lem:vaf} using the same argument as in the first statement.
\end{enumerate}
\end{remarks}

\section{Subgroups of $\A2$ admitting multiple free-by-free extensions} \label{sec:mul}

In the proof of Theorem \ref{thm:kw} we have seen that $K_4 = P_4/Z(P_4) \fin \A2$ can be written as an extension 
\begin{equation} \label{eq:jext} 1 \longrightarrow J \longrightarrow K_4 \stackrel{\psi}\longrightarrow \G[2] \longrightarrow 1. \end{equation}
In this section we will try to understand the nature of this extension and to relate it with other, more familiar, results. 

The starting point is the observation that the quotient of $B_4$ by its normal subgroup  $F_2 \cong \langle \sigma_1 \sigma_3^{-1},\sigma_2 \sigma_1 \sigma_3^{-1}\sigma_2^{-1}\rangle \unlhd B_4$ is $B_3$, the braid group on three strands. This can be seen by adding the relator $\sigma_1 \sigma_3^{-1}$ to the presentation of $B_4$ in Equation  (\ref{eq:presb4}), and noticing that this yields the Artin presentation of $B_3$. (The normal closure of $\sigma_1 \sigma_3^{-1}$ coincides with $F_2$.) We denote by $\Psi \colon B_4 \to B_3$ this map and observe, for future reference, that by explicit calculation the image of the generator $(\sigma_1 \sigma_2 \sigma_3)^4$ of the center $Z(B_4)$
is mapped to the {\em square} of the generator $(\sigma_1 \sigma_2)^3$ of the center $Z(B_3)$. The map $\Psi$ is referred to in literature  as ``Cardano--Ferrari" (see e.g. \cite{NYT10}) or ``canonical epimorphism" (see e.g. \cite{Lin04}). Using again the fact that $\Xi(F_2) = \op{V}_4 \unlhd \Sy4$, the epimorphism $\Xi \colon B_4 \to \op{S}_4$ descends to an epimorphism from $B_3 \to \op{S}_3$: note that such epimorphism is unique up to automorphisms of $\op{S}_3$, hence its kernel is $P_3$. We can then ``lift" the commutative diagram in Equation (\ref{eq:pure}) to the commutative diagram 
\begin{equation} \label{eq:lift} \xymatrix@R0.5cm{
& 1 \ar[d] &
1 \ar[d] & 1 \ar[d] &\\
1 \ar[r] & J \ar[d] \ar[r] &
P_4 \ar[d] \ar[r]^{\Psi}& P_3 \ar[d] \ar[r] & 1 \\
1\ar[r]&
F_2 \ar[r]\ar[d]^{\Xi}&
B_4 \ar[d]^{\Xi} \ar[r]^{\Psi}& B_3 \ar[d] \ar[r]&1\\
 1\ar[r]& \op{V}_4 \ar[d]
 \ar[r]& \Sy4 \ar[r] \ar[d]  &
\op{S}_3 \ar[r] \ar[d] &1 \\
& 1 & 1 & 1} 
\end{equation}

Note that $\Psi \colon P_4 \to P_3$ is referred to in \cite[Lemma 2.9]{KMM15} as the ``unusual map", and this map {\em is not} a strand removing map.  (While we denote its kernel 
$J = \ker{\Psi} \unlhd P_4$ with the same symbol devoted to $\ker{\psi} \unlhd K_4$ it is important to keep in mind that these are isomorphic but distinct objects.)  We can now ``shelf" the commutative diagrams in Equations (\ref{eq:pure}, \ref{eq:lift}) and join them in a commutative diagram:
\[ \xymatrix@=12pt{ 
& & J \ar[dr] \ar@{-}'[d][dd] \ar[rrr] & & & J \ar[dr] \ar@{-}'[d][dd] & &
\\ Z(P_4)  \ar[rrr]\ar[ddd] \ar[dr] &   & & P_4 \ar'[d][ddd] \ar[dr]^{\Psi}  \ar[rrr] & &  & K_4 \ar[dr]^{\psi} \ar'[d][ddd] & 
\\ & 2Z(P_3) \ar[ddd] \ar[rrr] & \ar[d] & &  P_3 \ar[rrr]\ar[ddd] & \ar[d] & & \G[2] \ar[ddd]
\\ &      & F_2 \ar@{-}'[d][dd] \ar[dr] \ar[dr] \ar@{-}'[r][rr] & & \ar[r]& \IA2 \ar@{-}'[d][dd] \ar[dr] & &  &
\\ Z(B_4)  \ar'[r][rrr]\ar[dr] &  &  & B_4 \ar'[d][ddd]  \ar'[r][rrr]\ar[dr]^{\Psi} & &  & \SA2 \ar'[d][ddd] \ar[dr]^{\psi} &
\\ & 2Z(B_3) \ar[rrr]   & \ar[d] & & B_3 \ar[ddd] \ar[rrr] &  \ar[d] &  & SL_2(\Z) \ar[ddd] & 
\\ & &\op{V}_4 \ar[dr] \ar@{-}'[r][rr] &  & \ar[r] & \op{V}_4 \ar[dr] & & &
\\  & &  & \Sy4 \ar'[r][rrr] \ar[dr]  & & & \Sy4 \ar[dr] & &
\\  & &  & & \op{S}_3 \ar[rrr] & &  & \op{S}_3 & } \]
(See \cite[Section 3] {CS24} for a geometric interpretation of parts of this diagram.)
Here, all $2$- and $3$-term sequences are exact.  Regarding the horizontal sequence on the front, we recognize that $1 \to 2Z(B_3) \to B_3 \to SL_2(\Z) \to 1$ is the universal central extension of $SL_2(\Z)$ (see \cite[Section 1.3.6.4]{FM12}), while, recalling that $K_3 = P_3/Z(P_3) \cong F_2$, we can recast the fact that $\G[2] = P_3/2Z(P_3) \cong F_2 \times \Z_2$.  

Now, the pure braid group $P_4$ can be written, for all choice of $1 \leq i \leq 4$, as an extension \[ 1 \longrightarrow F_3 \longrightarrow P_4 \stackrel{\Theta_i}{\longrightarrow} P_3 \longrightarrow 1 \] where the epimorphism $\Theta_i \colon P_4 \to P_3$ corresponds to the removal of the $i$-th stand. (This sequence can be seen as a consequence of the Fadell--Neuwirth Theorem in \cite{FN62}.) The kernel $F_3$ of this epimorphism consists of braids that become trivial removing that strand. (We will see later that for different $i$'s these kernels are $4$ distinct subgroups of $P_4$, as one would expect for geometric reasons.) 

One may wonder if there is a relation between the normal subgroups $J = \ker{\Psi} \unlhd P_4$ and $\ker{\Theta_i} \unlhd P_4$. We will learn more about this, but we can already exclude a simple relationship: If we had $J \fin \ker{\Theta_i}$ for some choice of $\Theta_i$, we would have 
\[ \xymatrix@R0.5cm{
& & 1 \ar[d] & &\\
 & & F_3 \ar[d] &   &  \\
1\ar[r]& J \ar[r] \ar@{^{(}-->}[ur] & P_4\ar[d]_{\Theta_i} \ar[r]^{\Psi} & P_3 \ar@{-->>}[dl] \ar[r]&1\\
& & P_3 \ar[d]  & & \\
& & 1 & } \]
as $P_3$ is Hopfian, an epimorphism to itself is an isomorphism, which would then imply $F_3 \cong \ker{\Theta_i} = J \cong F_5$. Note that this entails that $P_4$ admits inequivalent poly-free filtrations with factors of different rank. (The mild discomfort that two epimorphisms between the same two groups may have non-isomorphic kernels should be assuaged by the thought that this happens already for most epimorphisms from $P_3 \cong F_2 \times \Z$ to $P_2 \cong \Z$: all but three of them are not strand-forgetting maps.)

Next, we want to relate the strand-forgetting maps on $P_4$ to extensions on $K_4$. The epimorphisms $\Theta_i \colon P_4 \to P_3$ (unlike $\Psi$) restrict to an isomorphism between the centers $Z(P_4)$ and $Z(P_3)$, so they descend to epimorphisms $\theta_i \colon K_4 \to K_3$  and we have the commutative diagrams, for $1 \leq i \leq 4$,
\begin{equation} \label{eq:newquot} \xymatrix@R0.5cm{
&  & 1 \ar[d] & 1 \ar[d] &\\
& 1 \ar[r] \ar[d] & Z(P_4) \ar[d] \ar[r] & Z(P_3) \ar[d] \ar[r] & 1 \\
1\ar[r]&
F_3 \ar[r]\ar[d]^{\cong} &
P_4 \ar[d] \ar[r]^{\Theta_i} & P_3 \ar[d] \ar[r]&1\\
1 \ar[r] & F_3
 \ar[r] \ar[d] & K_4 \ar[r]^{\theta_i} \ar[d]  &
K_3 \ar[r] \ar[d] &1 \\
& 1 & 1 & 1}
\end{equation}

It follows from the bottom sequence in Equation (\ref{eq:newquot}) that $K_4$ can be written as free-by-free group in four different ways (that we refer to, a bit improperly, as strand-forgetting).
However, as mentioned in Section \ref{sec:virt}, one can identify $K_4$ with $\op{PMod}(S_{0,5})$. It follows that one should expect to inherit (at least) five inequivalent extensions of $K_4$ as $F_3$-by-$F_2$ group, arising from point-forgetting maps in the Birman  sequence for $\op{PMod}(S_{0,5})$, as $\op{PMod}(S_{0,4}) \cong F_2$. 

In what follows, we will recast this fifth extension, and relate it to the Cardano--Ferrari map $\Psi \colon P_4 \to P_3$. The first step will consist in identifying, for purely algebraic reason, the candidate fifth $F_3$-by-$F_2$ structure on $K_4$ out of extension in Equation (\ref{eq:jext}). 

\begin{lemma} \label{lem:newf3} The kernel $\Pi$ of the epimorphism $K_{4} \to \G[2] \cong F_2 \times \Z_2 \to F_2$ is a free group $\Pi \cong F_3 \unrhd_{f} J$, endowing $K_4$ with a $F_3$-by-$F_2$ structure. \end{lemma} 
\begin{proof} As $K_4$ is free-by-free, it has cohomological dimension $2$. The kernel $\Pi$ contains $J$ as subgroup of index $2$, hence it is finitely presented, and it is a normal subgroup of infinite index in a group of cohomological dimension $2$. By \cite[Theorem B]{Bi76} it is therefore free, and as $\chi(K_4) = 2$, its rank must equal $3$.  \end{proof}

 In order to continue, we need to better understand the isomorphism between $B_4/Z(B_4)$ and a subgroup of $\op{Mod}(S_{0,5})$. We will recall some details of this isomorphism, referring the reader to \cite[Section 1.3]{BB05} or \cite[Section 1.9]{FM12} for a more thorough discussion. We will focus on the case of the braid group with $4$ strands, but what we write applies to any number of strands. To start, there is an isomorphism $B_4 \cong \op{Mod}(D_4)$, the mapping class group of a $4$-punctured closed disk $D$, fixing the boundary pointwise and the punctures set-wise. The elements of the latter, thought of as isotopy classes of selfhomeomorphism of the disk, are nullisotopic. The graph of the nullisotopy (thought of as a map from $D \times I$ to $D$ which is constant on the boundary of the disk) carries the set of punctures in the interior of the disk back to itself, defining this way a braid in which each puncture traces a braid's strand. Next, the generator of the center $Z(B_4)$ of the braid group is identified with the class, in $\op{Mod}(D_4)$, represented by a Dehn twist about the boundary of $D$. Let $S_{0,5}$ be the sphere obtained by capping $D_4$ with a once-punctured disk; this way $S_{0,5}$ acquires a distinguished puncture, that we denote with the symbol $\ast$, arising from the capping operation. The inclusion-induced homomorphism from $\op{Mod}(D_4)$ to $\op{Mod}(S_{0,5})$ has kernel equal to the cyclic group generated by that Dehn twist, and image equal to the (index $5$) subgroup $\op{Mod}(S_{0,4},\ast)$ of $\op{Mod}(S_{0,5})$ fixing the distinguished puncture (see \cite[Section 1.4.2.5]{FM12}). This informs us that the aforementioned isomorphism of $B_4/Z(B_4)$ with a subgroup of $\op{Mod}(S_{0,5})$ bundles separately the distinguished puncture $\ast$ of $S_{0,5}$ and the remaining four punctures, with the latter corresponding with the four strands of the braid group. Now, it is not hard to convince oneself that the epimorphisms $\xi \colon B_4/Z(B_4) \cong \SA2 \to \op{S}_4$  and the epimorphism from $\op{Mod}(S_{0,5})$ to $\op{S}_{5}$ determined by the permutation of {\em all $5$} punctures fit in the commutative diagram
\begin{equation} \xymatrix@R0.5cm{
1 \ar[r] & K_4 \ar[d]_{\cong} \ar[r]  & \SA2
\ar[d]^{\cong} \ar[r]^{\hspace*{12pt} \xi} & \op{S}_4 \ar[d]^{\cong} \ar[r] & 1 & \\
1\ar[r] &
\op{PMod}(S_{0,5})  \ar[r]\ar[dr]_{\cong} &
\op{Mod}(S_{0,4},\ast) \ar[r] \ar[dr]  & \op{S}_4 \ar[r] \ar[dr] & 1 & \\
  & 1 \ar[r] & \op{PMod}(S_{0,5})  \ar[r] &  \op{Mod}(S_{0,5}) \ar[r] & \op{S}_5 \ar[r] & 1} 
\end{equation}
(See e.g. \cite[Equation (2.8)]{MN22}.) Now $\op{Mod}(S_{0,4},\ast)$ admits a point-forgetting Birman sequence, as in \cite[Theorem 1.4.6]{FM12}, giving
\begin{equation} \label{eq:birast} 1 \longrightarrow \pi_1(S_{0,4},\ast) \longrightarrow \op{Mod}(S_{0,4},\ast) \stackrel{\beta}{\longrightarrow} \op{Mod}(S_{0,4}) \longrightarrow 1 \end{equation}
This endows $\SA2$ of a normal subgroup $\pi_1(S_{0,4},\ast) \cong F_3 \unlhd \SA2$. At first glance, one may guess that it may be a (normal, index $2$) subgroup of $\IA2 \cong F_2$, but we can quickly dispel that guess: as we have previously discussed, the image of $\IA2$ in $\op{S}_4$ is the Klein subgroup $V_4 \unlhd \op{S_4}$, while the image of the point-pushing homeomorphisms $\pi_1(S_{0,4},\ast) \unlhd \op{Mod}(S_{0,4},\ast)$ in $\op{S}_4$ is necessarily trivial, as point-pushing homeomorphisms cannot move punctures. (Stated otherwise, $\pi_1(S_{0,4},\ast)$ is contained in the pure mapping class group.) However, we will prove that the guess is {\em almost} true. In order to do so, it will be convenient to refer to Artin's presentation of the pure braid group, which induces (much as in Equation (\ref{eq:presaut})) a presentation for $K_4$:  Artin's presentation  shows that $P_n$ is generated by the $n  \choose 2$ generators \[ A_{pq} := \sigma_{q-1} \ldots \sigma_{p+1} \sigma^2_p \sigma^{-1}_{p+1} \ldots \sigma^{-1}_{q-1}, \,\ 1 \leq p < q \leq n; \] for $n=4$ this gives six generators. 

\begin{lemma} \label{lem:fifth} The subgroup $\pi_1(S_{0,4},\ast) \unlhd K_4 \unlhd_f \SA2$ has image $\Z_2 \unlhd  \G[2] \fin SL_2(\Z)$ under the map $\psi \colon \SA2 \to SL_2(\Z)$. Consequently, $\pi_1(S_{0,4},\ast)$ coincides with the $F_3$-subgroup $\Pi \unlhd K_4$ defined in Lemma \ref{lem:newf3} and $\Pi \cap \op{Inn}(F_2) = J \unlhd K_4$. 
\end{lemma}
\begin{proof} In order to prove the first part of the statement, we need to study the action of the elements of $\pi_1(S_{0,4},\ast)$ on the homology of $\op{Inn}(F_2) \cong F_2 = \langle \s1\u3,\s2\s1\u3\u2 \rangle$. We can choose, as generators of the point-pushing subgroup, (the based homotopy class of) three disjoint loops $\ell_i, \, 2 \leq i \leq 4$ in $S_{0,5}$ based at the distinguished puncture $\ast$ and circling once counterclockwise one and only one of three of the remaining puntures, say $p_4,p_3,p_2$ (the loop around of the fourth puncture $p_1$ is then the inverse of the product of the other three).  Now a 
point-pushing homeomorphism can be described, in terms of Dehn twists (see \cite[Section 1.4.2.2]{FM12}), as product of {\em the Dehn twist along the simple closed curve obtained by pushing the loop off itself to the right} (defined with respect of the orientation of the loop), towards the puncture $p_i$, and {\em the inverse of the Dehn twist along the simple closed curve obtained by pushing the loop off itself to the left}, away from the puncture $p_i$. The first simple closed curve bounds a puntured disk (with puncture $p_i$), and is nullhomotopic to the puncture, hence the Dehn twist along that curve is nullisotopic. The second simple closed curve, instead, bounds a disk containing the puncture $p_i$ and the distinguished puncture $\ast$. Based on that, the inverse of the Dehn twist has the effect of fully braiding the  traces of the puncture $p_i$ and of the distinguished puncture $\ast$, with a {\em counterclockwise} motion (i.e. the inverse of the standard generator of the pure braid group on two strands). Figure (\ref{fig:sphere}) shows the strand corresponding to the puncture $p_3$, as it winds counterclockwise around the distinguished puncture $\ast$. 
\begin{figure}[h!]
  \includegraphics[scale=0.35]{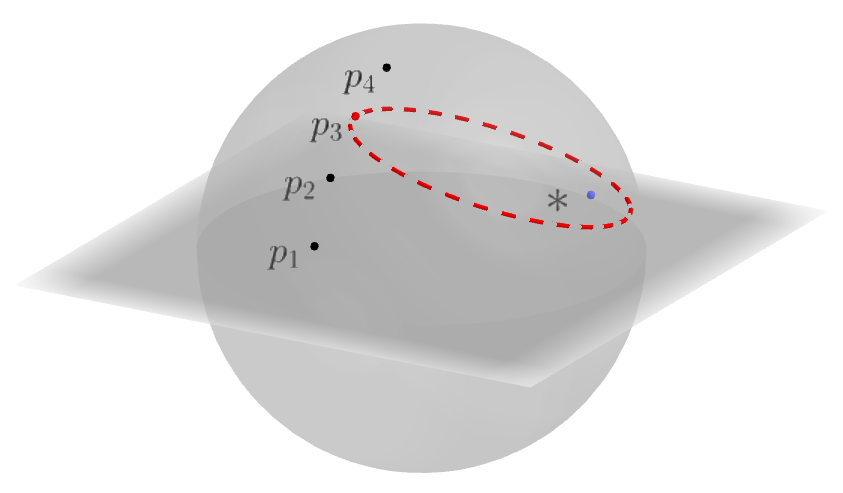}
  \caption{The strand traced by $p_3$ as it winds around $\ast$.}
  \label{fig:sphere}
\end{figure}
In this sense, we can think of Figure (\ref{fig:sphere}) as describing level sets of the graph of a nullisotopy of $S_{0,5}$ fixing pointwise all but the $3$rd puncture, with the $4$ strands arising from the image of the punctures $p_i$ along the time coordinate, with three strands constant and the dashed strand representing the motion of $p_3$. (The full graph could be described using a spherical shell whose radial coordinate represents time.)

Next, sliding the strand corresponding to the punture $p_i$ along $S_{0,5}$ has the effect of having the $i$-th strand wind in clockwise direction around the other three strands. See Figure (\ref{fig:threebraids}) for a description of the resulting braids for $i=4$, $i=3$ and $i=2$ respectively, composing bottom-to-top.
\begin{figure}[h!] 
\begin{center}
\begin{tikzpicture}[scale=1.0] 
\pic[
xscale=-1,
rotate=180,
braid/.cd,
number of strands=4,
every strand/.style={ultra thick},
strand 4/.style={red},
strand 3/.style={black},
strand 2/.style={black},
strand 1/.style={black}
] at (0,0) {braid={s_3^{-1} s_2^{-1} s_1^{-1} s_1^{-1} s_2^{-1} s_3^{-1}}};
\node[below right,yshift=-2mm,scale=0.98] at (-1.2,0) {$\ell_4 = \s3\s2\s1^2\s2\s3=A_{14}A_{24}A_{34}$};
\pic[
xscale=-1,
rotate=180,
braid/.cd,
number of strands=4,
every strand/.style={ultra thick},
strand 4/.style={black},
strand 3/.style={red},
strand 2/.style={black},
strand 1/.style={black}
] at (5.2,0) {braid={s_2^{-1} s_1^{-1} s_1^{-1} s_2^{-1} s_3^{-1} s_3^{-1}}};
\node[below right,yshift=-2mm,scale=0.98] at (4.2,0) {$\ell_3 = \s2\s1^2\s2\s3^2 = A_{13}A_{23}A_{34}$};
\pic[
xscale=-1,
rotate=180,
braid/.cd,
number of strands=4,
every strand/.style={ultra thick},
strand 4/.style={black},
strand 3/.style={black},
strand 2/.style={red},
strand 1/.style={black}
] at (10.4,0) {braid={s_1^{-1} s_1^{-1} s_2^{-1} s_3^{-1} s_3^{-1} s_2^{-1}}};
\node[below right,yshift=-2mm,scale=0.98] at (9.2,0) {$\ell_2 = \s1^2\s2\s3^3\s2 = A_{12} A_{23} A_{24}$};
\end{tikzpicture}
\caption{Braids corresponding to point-pushing maps.}
\label{fig:threebraids}
\end{center}
\end{figure}

The element $\ell_i$, thought of as an element of $\op{Mod}(S_{0,4},\ast)$, is determined by the choice of any braid of $B_4 \cong \op{Mod}(D_4)$ which maps to $\ell_i$ under the epimorphism $\op{Mod}(D_4) \to \op{Mod}(S_{0,4},\ast) \cong \SA2$. 
Any of these choices differ by a central element of the braid group. Stated differently, the word in the $\sigma_j$'s (or in the $A_{pq}$'s) describing $\ell_i$ can be thought of as describing an element in $\SA2$, via the presentation in Equation (\ref{eq:presaut}), or the similar presentation for $K_4$. 

Here are the details of the equivalence of the two presentations for the $\ell_i$'s:
\begin{equation} \label{eq:twopres}
\begin{split}
\ell_4 & = \s3\s2\s1^2\s2\s3 = \s3\s2\s1^2\u2\u3\s3\s2^2\u3\s3^2 = A_{14} A_{24} A_{34} \\
\ell_3 & = \s2\s1^2\s2\s3^2 = \s2\s1^2 \u2\s2^2\s3^2 = A_{13}A_{23}A_{34} \\
\ell_2 & = \s1^2\s2\s3^2\s2 = \s1^2 \s2\s3^2\s2\s3\u3 = \s1^2 \s2\s3\underbrace{\s3\s2\s3}\u3 = \s1^2 \s2 \mathrlap{\overbrace{\phantom{\s3\s2\s3}}}
\s3\underbrace{\s2\s3\s2}\u3 = \\
& =  \s1^2 \s2 \overbrace{\s2\s3\s2} \s2\u3 = \s1^2 \s2^2 \s3 \s2^2 \u3 = A_{12}A_{23}A_{24} 
\end{split} 
\end{equation}

At this point the action of $\ell_i$ on $\op{Inn}(F_2)$ is determined by the conjugation action of the braid identified in Equation (\ref{eq:twopres} on the subgroup $\langle \s1\u3,\s2\s1\u3\u2 \rangle \unlhd B_4$. Using the formulae for the image of the $A_{pq}$'s in $SL_2(\Z)$ determined in the Appendix we get
\begin{align*} & \psi(\ell_4) =   \begin{pmatrix} -1 & 2 \\ -2 & 3 \end{pmatrix}  \begin{pmatrix} 1 & 2 \\ 0 & 1 \end{pmatrix} \begin{pmatrix} 1 & 0 \\ -2 & 1 \end{pmatrix} = \begin{pmatrix} -1 &  0 \\ 0 & -1 \end{pmatrix} \\ 
& \psi(\ell_3) =   \begin{pmatrix} -3 & 2 \\ -8 & 5 \end{pmatrix}  \begin{pmatrix} -1 & 2 \\ -2 & 3 \end{pmatrix} \begin{pmatrix} 1 & 0 \\ -2 & 1 \end{pmatrix} = \begin{pmatrix} -1 & 0  \\ 0 & -1 \end{pmatrix} \\ 
& \psi(\ell_2) =   \begin{pmatrix} 1 & 0 \\ -2 & 1 \end{pmatrix}  \begin{pmatrix} -1 & 2 \\ -2 & 3 \end{pmatrix} \begin{pmatrix} 1 & 2 \\ 0 & 1 \end{pmatrix}  = \begin{pmatrix} -1 &  0 \\ 0 & -1 \end{pmatrix}
\end{align*} which shows that the image of the point-pushing subgroup in $SL_{2}(\Z)$, being torsion, is in fact necessarily the $\Z_2$ subgroup of the principal congruence subgroup $\G[2]$. To complete the prove of the statement, notice that as $\pi_1(S_{0,4},\ast)$ is normal in $K_4$ and it is contained in the kernel $\Pi$ of the epimorphism $K_{4} \to \G[2] \cong F_2 \times \Z_2 \to F_2$, it is a normal subgroup of $\Pi$. Finitely generated normal subgroups of free groups have finite index, and as the ranks coincide we must have $\Pi \cong \pi_1(S_{0,4},\ast)$. Finally, the intersection $\Pi \cap \op{Inn}(F_2) \unlhd K_4$ can be thought of as the collection of elements of $K_4$ that have trivial image under $\psi \colon K_4 \to \G[2]$, and that's exactly $J \unlhd K_4$.
\end{proof}

It is worth mentioning that Lemma \ref{lem:fifth} establishes the relationship between the kernel of the Birman sequence of Equation (\ref{eq:birast}) for  $\op{Mod}(S_{0,4},\ast)$ with the kernel of the Birman sequence for $\op{PMod}(S_{1,2})$. In fact, as we mentioned before we have an isomorphism $\op{Mod}(S_{0,4},\ast) \cong \SA2 \cong 
\op{PMod}(S_{1,2})$. If we denote by $q_1,q_2$ the punctures on the torus, we have {\em two} Birman sequences associated with a choice of either puncture: 
\begin{equation} 1 \longrightarrow \pi_1(S_{1,1},q_i) \longrightarrow \op{PMod}(S_{1,2}) \stackrel{\beta_i}{\longrightarrow} \op{Mod}(S_{1,1}) \longrightarrow 1 \end{equation} 
\begin{figure}[h!]
  \includegraphics[scale=0.35]{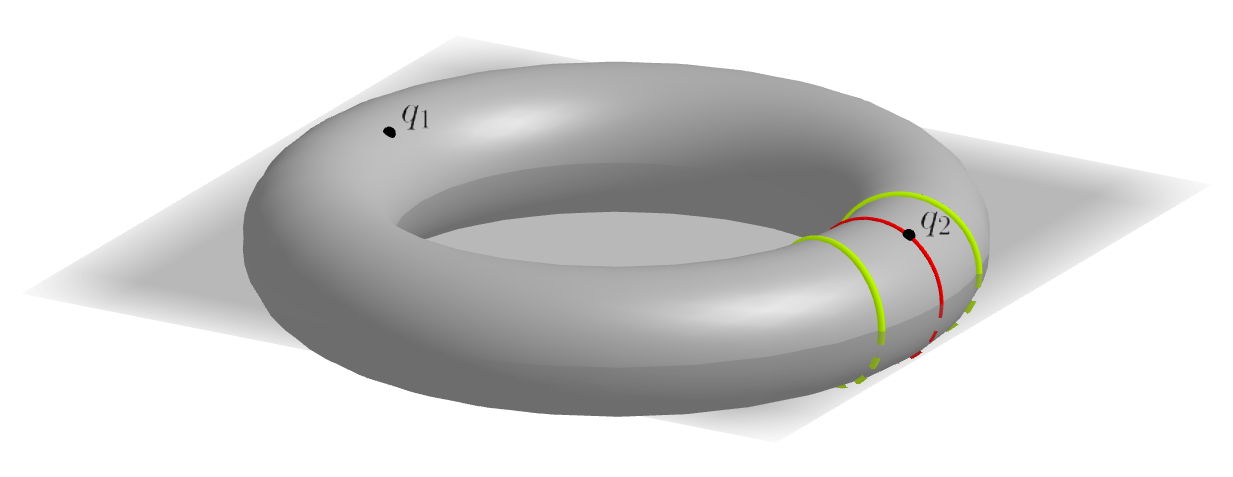}
  \caption{The simple closed curves determining a point-pushing homeomorphism.}
  \label{fig:torus}
\end{figure}
The relation between $\ker{\b_i} \cong F_2$ and $\Pi$ is described in the following:

\begin{corollary} The subgroups $\pi_1(S_{1,1},q_i) \unlhd \op{PMod}(S_{1,2})$, $i=1,2$ coincide, and the isomorphisms  $
\op{PMod}(S_{1,2}) \cong \SA2$, $\op{Mod}(S_{0,4},\ast) \cong \SA2$ map $\pi_1(S_{1,1},q_i)$ to $\op{Inn}(F_2)$ and $\pi_1(S_{0,4},\ast)$ to $\Pi$ respectively.
\end{corollary} 
\begin{proof} The second part of this corollary is simply a rephrasing of Lemma \ref{lem:fifth} (together with the fact that the isomorphism $\op{PMod}(S_{1,2}) \cong \SA2$ is defined exactly in terms of the adjoint action of $\op{PMod}(S_{1,2})$ on the Birman kernel). What is not included in there is the fact that the choice of the Birman sequence is immaterial, and this is due to the fact that the subgroups of $\op{PMod}(S_{1,2})$ represented by point-pushing homeomorphisms do not depend by the choice of the punctures. (This fact is likely to be well-known to the experts.) In fact, once standard loop representatives for the generators of $\pi_1(S_{1,1},q_i)$ are chosen, the corresponding point-pushing homeomorphisms can be represented by composition of Dehn twists or their inverses along simple closed curves obtained by pushing the base loops (as depicted in Figure (\ref{fig:torus}) for one of the generators) right and left respectively.  
It should be quite clear that we can isotope these curves so that they play (reverting left and right) the role of the simple closed curves obtained by pushing the loop representative based at the other puncture. The point pushing groups have therefore the same generating sets irrespective of the choice of $q_i$.
\end{proof}

Note that the isomorphism of the point-pushing subgroups discussed here is an accident due to the number of punctures and, as we are about to see, fails to hold in general. 

We are in position to state the following:

\secthm*

\begin{proof} Equation \ref{eq:birast} describes $\SA2$ as an extension with fiber $F_3 \cong \pi_1(S_{0,4},\ast)$ (and base $\op{Mod}(S_{0,4})  \cong (\Z_{2} \times \Z_{2}) \rtimes PSL_{2}(\Z)$, see \cite[Proposition 1.2.7]{FM12} for the latter identification). While $\SA2 \cong \op{Mod}(S_{0,4},\ast)$ admits the unique Birman sequence of Equation (\ref{eq:birast}), its finite index subgroup $K_4 \cong \op{PMod}(S_{0,5})$ admits five, adding the point-forgetting ones, where each puncture $p_i$, $1 \leq i \leq 4$ assumes the role of base point. These latter four arise, by the discussion in the proof of Lemma \ref{lem:fifth}, from the strand-forgetting maps $\Theta_i \colon P_4 \to P_3$, see Equation (\ref{eq:newquot}), while the first is related to the Cardano--Ferrari map $\Psi \colon P_4 \to P_3$, see Lemma \ref{lem:newf3}. 

We conclude with the proof that all these $F_3$-by-$F_2$ structures on $K_4$ are inequivalent, in the sense that the fiber groups are distinct normal subgroups. In fact, we will show more, namely that the Frobenius product of any two fiber groups generates the whole $K_4$. This affords a direct (and slightly cumbersome) proof based on the presentation for $K_4$ with generators $A_{pq}$, but we will provide a more direct proof. Recall that the epimorphisms $\Theta_i \colon P_4 \to P_3$ have kernel the free group generated by the three generators $A_{pq}$'s that have one index equal to $i$, and are the identity on the remaining three (up to suitable relabeling), see e.g. \cite{BB05}. The same applies, with the induced presentations, for the epimorphisms $\theta_i \colon K_4 \to K_3$. Now take, say, $\ker{\theta_4} = \langle A_{14},A_{24},A_{34} \rangle$ and $\ker{\theta_3} = \langle A_{13},A_{23},A_{34} \rangle$. We have the diagram
\[ \xymatrix@R0.5cm{
& 1  \ar[d] & 1 \ar[d] & 1 \ar[d] &\\
1 \ar[r] & F_{\omega} \ar[d] \ar[r] &
\langle A_{13}, A_{23}, A_{34} \rangle \ar[d] \ar[r] & K_3 \ar[d] \ar[r] & 1 \\
1\ar[r]&
\langle A_{14},A_{24},A_{34} \rangle  \ar[r]\ar[d] &
K_4 \ar[d]^{\theta_3} \ar[r]^{\theta_4} & K_3 \ar[d] \ar[r]&1\\
1 \ar[r] & K_3 \ar[d]
 \ar[r]& K_3 \ar[r] \ar[d]  &
 1  & \\
&1  & 1 &}
\] 
where surjectivity the map $\theta_4 \colon \langle A_{13}, A_{23}, A_{34} \rangle  \to K_3$ comes from we study the image of the generators: we have 
\begin{align*} 
\theta_4 \colon & A_{13} \mapsto A_{13} & A_{23} & \mapsto  A_{23} & A_{34} \mapsto  1, 
\end{align*} and recalling that 
\[ K_3 = \langle A_{12},A_{13},A_{23}|A_{12}A_{13}A_{23} \rangle \cong F_2 \] 
it is immediate that $\theta_4$ is surjective. Now the index of $\theta_4( \ker \theta_3) \fin K_3$ equals the index of the Frobenius product $\ker \theta_3 \cdot \ker \theta_4 \fin K_4$ (and the latter equals simultaneously the index of $\theta_3( \ker \theta_4) \fin K_3$, namely $1$, $\theta_3$ must be surjective as well).  The same argument applies as well for the remaining pairs of strand-forgetting maps $\theta_i$.
 
Next, consider $\Pi = \ker \b = \langle \ell_2, \ell_3, \ell_4 \rangle$. By symmetry, we expect the same result as above, and in fact we have

\[ \xymatrix@R0.5cm{
& 1  \ar[d] & 1 \ar[d] & 1 \ar[d] &\\
1 \ar[r] & F_{\omega} \ar[d] \ar[r] &
\langle \ell_2, \ell_3, \ell_4 \rangle \ar[d] \ar[r] & K_3 \ar[d] \ar[r] & 1 \\
1\ar[r]&
\langle A_{14},A_{24},A_{34} \rangle  \ar[r]\ar[d] &
K_4 \ar[d]^{\b} \ar[r]^{\theta_4} & K_3 \ar[d] \ar[r]&1\\
1 \ar[r] & K_3 \ar[d]
 \ar[r]& K_3 \ar[r] \ar[d]  &
 1  & \\
&1  & 1 &}
\] 
where surjectivity of the map $\theta_4 \colon \langle \ell_2, \ell_3, \ell_4 \rangle  \to K_3$ comes again from the study of the image of the generators: we have 
\begin{align*} 
\theta_4 \colon & \ell_2 = A_{12}A_{23}A_{24} \mapsto A_{12}A_{23} & \ell_3 = A_{13}A_{23}A_{34} & \mapsto  A_{13}A_{23} & \ell_4 = A_{14}A_{24}A_{34} \mapsto  1 
\end{align*} 
and with this in hand one can explicitly verify that $\theta_4$ is surjective. The rest of the argument follows as above.
\end{proof}

\begin{remark} All the five structure discussed in Theorem \ref{thm:sec} are related by homeomorphisms of $S_{0,5}$ that preserve the punctures set-wise. This observation reconciles, to an extent, the strand-forgetting maps  and the Cardano--Ferrari map from $P_4$ to $P_3$, which at face value seem to bear little relation, showing that they induce, on $K_4$, free-by-free structures related by an automorphism of $K_4$. (The asymmetry in the passage from $P_4 = \op{PMod}(D_4)$ to $\op{PMod}(S_{0,5})$ originates from the aforementioned fact that the capping naturally selects a distinguished puncture on $S_{0,5}$.) 
\end{remark}

\begin{question} Does there exist a sixth $F_3$-by-$F_2$ structure on $K_4$ which does not originate from the constructions above?
\end{question}

For sake of completeness, we want to analyze the relation between the normal subgroups $J$ and $\ker \Theta_i$ of $P_4$: this will bear some similarity and some differences with the result just proven. We have the following:

\begin{proposition} \label{pro:sec} The normal subgroups $J$ and $\ker{\Theta_i}$ of $P_4$ are distinct and their Frobenius products satisfy $J \cdot \Ker \Theta_i = P_4$, while $\ker \Theta_i  \cdot \Ker \Theta_j \unlhd P_4$ have infinite index.  
\end{proposition} 

\begin{proof} We start by describing $\Psi \colon P_4 \to P_3$ and the $\Theta_i \colon P_4 \to P_3$ in terms of Artin generators. Imposing the relation $\sigma_1 = \sigma_3$ on the generators, we can show by explicit computation that $\Psi \colon P_4 \to P_3$ is defined on the generators as \begin{align*} 
\Psi \colon  & A_{12} \mapsto A_{12} & A_{13} & \mapsto A_{13} &  A_{23} \mapsto A_{23}  \\ 
                & A_{14 }\mapsto A_{23} &  A_{24} & \mapsto A_{23}^{-1} A_{13} A_{23}  & A_{34} \mapsto A_{12} 
\end{align*} 
Two of these computations deserve some attention: we have 
\begin{align*}
\Psi(A_{14}) & = \Psi(\s3\s2\s1^2\u2\u3) = \underbrace{\s1\s2\s1}\overbrace{\s1\u2}\u1 = \underbrace{\s2\s1\s2}\overbrace{\u2\u1\s2\s1}\u1 = \s2^2 = A_{23} \\
\Psi(A_{24}) & = \Psi(\s3\s2^2\u3) = \underbrace{\s1\s2}\overbrace{\s2\u1} = \underbrace{\u2\s1\s2\s1}\overbrace{\u1\u2\s1\s2} = \u2 \s1^2 \s2 = \\ 
& = \s2^{-2} \s2\s1^2\u2\s2^2 = A_{23}^{-1} A_{13} A_{23}  
\end{align*}
With this information in hand, we can determine the images $\Psi(\ker{\Theta_i})$. 
For instance, for $i=4$, $\ker{\Theta_4}$ is generated by $A_{14},A_{24},A_{34}$; these generators map, under $\Psi$, into $A_{23},A_{23}^{-1} A_{13} A_{23},A_{12}$ respectively, hence they generate the full $P_3$; the kernel of $\Psi(\Ker\Theta_4) = \ker{\Theta_4} \cap J$ equals therefore the infinitely generated free group $F_{\omega}$ according to the commutative diagram
\[ \xymatrix@R0.5cm{
& 1  \ar[d] & 1 \ar[d] & 1 \ar[d] &\\
1 \ar[r] & F_{\omega} \ar[d] \ar[r] &
\ker{\Psi} \ar[d] \ar[r] & P_3 \ar[d] \ar[r] & 1 \\
1\ar[r]&
\langle A_{14},A_{24},A_{34} \rangle  \ar[r]\ar[d] &
P_4 \ar[d]^{\Psi} \ar[r]^{\Theta_4} & P_3 \ar[d] \ar[r]&1\\
1 \ar[r] & P_3 \ar[d]
 \ar[r]& P_3 \ar[r] \ar[d]  &
 1  & \\
&1  & 1 &}
\] Note that the surjectivity of $\Psi \colon \langle A_{14},A_{24},A_{34} \rangle \to P_3$ is true by inspection, and this entails as before the surjectivity of $\Theta_4 \colon \Ker{\Psi} \to P_3$.

The same argument applies {\em mutatis mutandis} to the other strand-forgetting maps $\Theta_i, i =1,2,3$. 

On the other hand, if we repeat the same argument for (say) $\Theta_3$ and $\Theta_4$, we see that \[ \Theta_4(\ker \Theta_3) = \langle A_{13}, A_{23} \rangle \unlhd P_3 \cong F_2 \times \Z \] and the latter inclusion is necessarily infinite index. (Equivalently, the subgroup of $P_4$ generated by $A_{14},A_{24},A_{34},A_{13},A_{23}$ is necessarily infinite index.) 
\end{proof} 

\begin{remark}
As we showed before, pairwise, $\ker{\Theta_i} \cap \ker{\Psi} = F_{\omega}$. Similarly, for $i \neq j$,  $\ker{\Theta_i} \cap \ker{\Theta_j} = F_{\omega}$ (as discussed in \cite[Section 7.2]{BCWW06}). Given these facts, we can also deduce that collectively $\bigcap_{i=1}^{4} \ker{\Theta_i}$ and $\ker{\Psi} \cap \bigcap_{i=1}^{4} \ker{\Theta_i}$ are infinitely generated free groups because of the nontriviality of the intersection which follows from \cite[Lemma 2.1]{Lo86}. The kernels $\ker{\Theta_i}$ have a clear geometric meaning: they compose the subgroup of $4$-strand pure braids which become trivial removing the $i$-th strand. They intersection is given by the group $\bigcap_{i=1}^{4} \ker{\Theta_i}$ of {\em Brunnian} $4$-strand braids, which satisfy such property for the removal of any strand. (The braid yielding the Borromean rings give the simplest nontrivial example of such class, albeit for $3$ strands).  The geometric meaning of the pure braids in $\ker{\Psi}$ is, to the best of our understanding, less obvious, and so is the meaning of $\ker{\Psi} \cap \bigcap_{i=1}^{4} \ker{\Theta_i}$. 
\end{remark}
\setcounter{secnumdepth}{-1}
\section{Appendix} 
We denote by $a := \s1\u3$ and $b:= \s2\s1\u3\u2$ two generators of the normal subgroup $F_2 \unlhd B_4$. The generators of the group $B_4$ and their inverses act by conjugation on $\langle a,b \rangle$ as follows:
\begin{align*} 
\s1 \colon  & a \mapsto a  & \s2 \colon & a \mapsto b & \s3 \colon & a \mapsto a  \\ 
                & b \mapsto ba^{-1}  & & b \mapsto ba^{-1}b &  & b \mapsto a^{-1}b \\
\u1 \colon  & a \mapsto a  & \u2 \colon & a \mapsto aba^{-1} & \u3 \colon & a \mapsto a  
\\               & b \mapsto ba  & & b \mapsto a &  & b \mapsto ab 
\end{align*} 
The second row deserves some attention, while the rest is either obvious or a straightforward consequence:
\begin{align*} \s1 b \u1 & = \underbrace{\s1\s2\s1}\u3\u2\u1 = \mathrlap{\underbrace{\phantom{\s2\s1\s2}}}
\s2\s1
\mathrlap{\overbrace{\phantom{\s2\u3}}}
\s2\u3\u2\u1 = \s2\s1\overbrace{\u3\u2\s3\s2}\u2\u1 = \\
& =  \s2\s1\u3\u2\s3\u1 = ba^{-1} \\
\s2 b \u2 & = \s2\underbrace{\s2\u3}\s1\u2\u2 = 
\s2 
\mathrlap{\underbrace{\phantom{\u3\u2\s3\s2}}}
\u3\u2\s3
\mathrlap{\overbrace{\phantom{\s2\s1\u2}}} \s2\s1\u2\u2 = 
\s2\u3\u2\s3\overbrace{\u1\s2\s1}\u2 = \\
& = \s2\u3\underbrace{\u2\u1}\s3\s2\s1\u2 = \s2\u3\mathrlap{\underbrace{\phantom{\s1\u2\u1\u2}}}
\s1\u2\u1
\mathrlap{\overbrace{\phantom{\u2\s3}}}\u2\s3\s2\s1\u2 = \\
& = \s2\s1\u3\u2\u1\overbrace{\s3\s2\u3\u2}\s2\s1\u2 = \s2\s1\u3\u2\s3\u1\s2\s1\u3\u2 = b a^{-1}b \\
\s3 b \u3 & = \s3 \underbrace{\s2\s1}\u3\u2\u3 = 
\s3
\mathrlap{\underbrace{\phantom{\u1\s2\s1\s2}}}
\u1\s2\s1
\mathrlap{\overbrace{\phantom{\s2\u3}}}
\s2\u3\u2\u3 = \\ 
& = \s3\u1\s2\s1\overbrace{\u3\u2\s3\s2}\u2\u3 = a^{-1}b 
\end{align*}
Using these results, and a fair amount of patience, allows one to verify that the action of the generators of the pure braid group on $\langle a,b \rangle$ is given by 
\begin{align*} 
A_{12} \colon  & a \mapsto a  & A_{13} \colon & a \mapsto b a^{-2} b a^{-1}  & A_{14} \colon & a \mapsto aba^{-1}ba^{-1}  \\ 
                & b \mapsto ba^{-2}  & & b \mapsto ba^{-2} ba^{-2} ba^{-1} ba^{-2} ba^{-1} &  & b \mapsto b^2 a^{-1}b a^{-1} \\
A_{23} \colon  & a \mapsto ba^{-1}b  & A_{24} \colon & a \mapsto ab^2 & A_{34} \colon & a \mapsto a  
\\               & b \mapsto ba^{-1}ba^{-1}b  & & b \mapsto b &  & b \mapsto a^{-2}b 
\end{align*}
Looking at the induced action of the generators $A_{pq}$ on the homology of $F_2 = \langle a,b\rangle$ one can recast explicitly the image of the generators of the pure braid group in $SL_2(\Z)$, that we write in matrix form using, for sake of notation, the same symbol:
\begin{align*} A_{12} = & \begin{pmatrix} 1 & 0 \\ -2 & 1 \end{pmatrix}  & A_{13} = & \begin{pmatrix} -3 & 2 \\ -8 & 5 \end{pmatrix} & A_{14} = & \begin{pmatrix} -1 & 2 \\ -2 & 3 \end{pmatrix} \\ 
 A_{23} = & \begin{pmatrix} -1 & 2 \\ -2 & 3 \end{pmatrix}  & A_{24} = & \begin{pmatrix} 1 & 2 \\ 0 & 1 \end{pmatrix} & A_{34} = & \begin{pmatrix} 1 & 0 \\ -2 & 1 \end{pmatrix}
\end{align*} which, in terms of the standard generators of $SL_2(\Z)$, namely
\begin{align*}  S = & \begin{pmatrix} 0 & -1 \\ 1 & 0 \end{pmatrix} & T = & \begin{pmatrix} 1 & 1 \\ 0 & 1 \end{pmatrix}
\end{align*}
can be written as
\begin{align*} A_{12} = & (TST)^{-2}  & A_{13} = & S^{-1}T^{-2}S^{-1}TS^{-1}T^{-2}ST^{-1} & A_{14} = & S^{-1}T^{-2}S^{-1}T^{-2} \\ 
 A_{23} = & S^{-1}T^{-2}S^{-1}T^{-2} & A_{24} = & T^2 & A_{34} = & (TST)^{-2}
\end{align*} (For an explanation of how to recover these equalities, see e.g. \cite{Con}.)
The map from $P_4$ to $SL_2(\Z) $ descends to $K_4$, as the center of $P_4$ acts trivially by conjugation on $F_2$.


\end{document}